\documentclass[11pt,leqno]{article}
\usepackage{amsmath, amsfonts,color}
\usepackage[latin1]{inputenc}
\usepackage[T1]{fontenc}
\usepackage[english]{babel}
\usepackage{lmodern}

\usepackage{mathrsfs}
\usepackage{amsmath}
\usepackage{amsthm}
\usepackage{amssymb}
\usepackage{mathabx}
\usepackage{latexsym}
\usepackage{graphicx}

\def\dfrac#1#2{\ds{\frac{#1}{#2}}}

\newcommand{\R}{{\mathbb R}}
\newcommand{\N}{{\mathbb N}}
\newcommand{\re}{{\mathbb R}}

\newcommand{\EEE}{{\mathbb E}}

\newcommand{\cT}{{\mathcal T}}
\newcommand{\cP}{{\mathcal P}}

\let\ds=\displaystyle

\def\R{{\mathbb R}}
\def\N{{\mathbb  N}}
\def\N{{\mathbb N}}
\def\ds{\displaystyle}

\numberwithin{equation}{section}

 \newtheorem{theorem}{Theorem}[section]
 \newtheorem{lemma}{Lemma}[section]
 \newtheorem{proposition}{Proposition}[section]
 \newtheorem{definition}{Definition}[section]
 \newtheorem{corollary}{Corollary}[section]
 \newtheorem{remark}{Remark}[section]

\DeclareMathOperator{\diver}{div}

\begin{document}
\title{Deterministic mean field games with control on the acceleration}

\author{{ Yves Achdou \thanks{
Universit{\'e} de Paris, Laboratoire Jacques-Louis Lions (LJLL), F-75013 Paris, France, achdou@ljll-univ-paris-diderot.fr}
 Paola Mannucci\thanks{Dipartimento di Matematica ``Tullio Levi-Civita'', Universit{\`a} di Padova, mannucci@math.unipd.it}}, \\
{Claudio Marchi \thanks{Dipartimento di Ingegneria dell'Informazione, Universit{\`a} di Padova, claudio.marchi@unipd.it},
  Nicoletta Tchou\thanks{Univ Rennes, CNRS, IRMAR - UMR 6625, F-35000 Rennes, France, nicoletta.tchou@univ-rennes1.fr}}}

\maketitle

\begin{abstract}
 In the present work, we study   deterministic mean field games (MFGs)  with finite time horizon in which the dynamics of a generic agent is controlled by the acceleration.  They are described by a system of PDEs
coupling a continuity  equation for the density of the distribution of states (forward in time) and a  Hamilton-Jacobi  (HJ) equation for the optimal value of a representative agent (backward in time).

The state variable is the pair $(x,v)\in \R ^N\times \R^N$  where $x$ stands for the position and $v$ stands for the velocity. The dynamics is often referred to as the {\sl double integrator}. In this case, the Hamiltonian of the system is neither strictly convex nor coercive, hence the available results on MFGs cannot be applied. Moreover, we will assume that the Hamiltonian is unbounded w.r.t. the velocity variable $v$. 
We prove the existence of a weak solution of the MFG system via a vanishing viscosity method and we characterize  the 
distribution of states as the image of the initial distribution by the flow associated with the optimal control.
\end{abstract}


\section{Introduction}
 The theory of  mean field games (MFGs for short) is more and more investigated since the pioneering works \cite{ll06-1,ll06-2,ll07} of  Lasry and Lions: it aims at studying the  asymptotic behaviour of   differential games (Nash equilibria) as the number  of agents tends to infinity.
 In the present work, we study   deterministic mean field games  with finite time horizon in which the dynamics of a generic agent is controlled by the acceleration.  They are described by a system of PDEs
coupling a continuity  equation for the density of the distribution of states (forward in time) and a  Hamilton-Jacobi  (HJ) equation for the optimal value of a representative agent (backward in time). The state variable is the pair $(x,v)\in \R ^N\times \R^N$  where $x$ stands for the position and $v$ stands for the velocity. 

The  systems of PDEs are of the form
\begin{equation}
\label{eq:MFGA}
\left\{
\begin{array}{rll}
(i)&-\partial_t u-v\cdot D_xu+H(x,v,D_vu)-F[m(t)](x,v)=0&\qquad \textrm{in }\re^{2N}\times (0,T)\\
(ii)& \partial_t m +v\cdot D_xm-{\rm div}_v(D_{p_v}H(x,v,D_vu)m)=0&\qquad \textrm{in }\re^{2N}\times (0,T)\\
(iii)& m(x,v,0)=m_0(x,v), u(x,v,T)=G[m(T)](x,v)\,,&\qquad \textrm{on }\re^{2N}
\end{array}\right.
\end{equation}
where $T$ is  a  positive real number, $u=u(x,v,t)$, $m=m(x,v,t)$, $(x,v)\in\R^{2N}$, $t\in(0,T)$ and
$H$ is defined by
\begin{equation}
\label{HamA}
H(x,v,p_v)=\max_{\alpha\in \R^N}(-\alpha p_v-l(x,v,\alpha)).
\end{equation}
We take $F$ and $G$ strongly regularizing and we assume that the running cost has the form $l(x,v,\alpha)=l(x,v)+\frac1 2 \vert \alpha\vert^2+\frac1 2 \vert v \vert^2$, where $(x,v)\mapsto l(x,v)$ is a bounded and $C^2$-bounded function.

{\it Formally}, systems of this form arise when the dynamics of the generic player is described by a {\sl double integrator}:
\begin{equation}\label{eq:HJA}
\left\{
\begin{array}{rcll}
 \xi'(s)&=&\eta(s),\quad   &s\in (t,T),\\
 \eta'(s)&=&\alpha(s),\quad   &s\in (t,T),\\
 \xi(t)&=&x, &\\\eta(t)&=&v,& 
\end{array}\right.
\end{equation}
and when the control law belongs to the space of the measurable  functions with values in $\R^N$ and is chosen in 
order to minimize the cost
\begin{equation}\label{cost}
J_t:=J_t(\xi,\eta, \alpha)
=\displaystyle\int_t^T
l(\xi(s), \eta(s), \alpha(s))+F[m(s)](\xi(s), \eta(s))ds+G[m(T)](\xi(T), \eta(T)).
\end{equation}
To summarize, the main features of this model are:
\begin{enumerate}
\item  The control $\alpha$ is only involved in the dynamics of the second component of the state variable, see  \eqref{eq:HJA}.
\item The running cost has the form
  \begin{equation}
    \label{eq:10}
l(\xi,\eta,\alpha)=l(\xi,\eta)+\frac1 2 \vert  \eta \vert^2+\frac1 2 \vert  \alpha \vert^2,
  \end{equation}
where $(\xi,\eta)\mapsto l(\xi,\eta)$ is a bounded $C^2$ function, thus the former is unbounded w.r.t. the variable $\eta$.
Note that $\vert  \eta \vert^2$ stands for a kinetic energy, whereas the term $\vert  \alpha \vert^2$ is a penalty for large accelerations.
Note also that the results of the present paper hold  for a fairly large class of generalizations of (\ref{eq:10}).
\item Setting $f(\xi, \eta, \alpha)= (\eta, \alpha)$, the Hamiltonian associated to the control problem of a generic player is 
  \begin{displaymath}
\mathcal{H}( \xi,\eta, p)= \max_{\alpha\in \R^N}\{-p\cdot f(\xi, \eta,\alpha) -l(\xi,\eta,\alpha)\}=-p_x\cdot \eta +H(\xi,\eta, p_v),    
  \end{displaymath}
where $p=(p_x,p_v)$ and $H$ is defined in (\ref{HamA}). The Hamilton $\mathcal{H}$ is neither  strictly convex  nor  coercive
with respect to $p=(p_x, p_v)$. Hence the available results 
on the regularity of the value function $u$ of the associated optimal control problem (\cite{CS}, \cite{Cla}, \cite{C})  and on the
existence of a solution of the MFG system (\cite{C}) cannot be applied.
\end{enumerate}
We recently learnt  that a similar type of mean field games has been studied in \cite{CM}, independently, at the same time, and with different techniques.  To the best of our knowledge, these systems have not been investigated elsewhere.

The main results of the present work are the existence of a solution of~\eqref{eq:MFGAs} and a characterization of the distribution of states~$m$. In order to establish the representation formula for~$m$, we use some ideas introduced by P-L Lions in his lectures at Coll\`ege de France (2012) (see \cite{L-coll,C}), some results proved in~\cite{CH,C13}, and the superposition principle \cite [Theorem 8.2.1]{AGS}. These methods rely on optimal control theory, in particular on optimal synthesis results.
In our setting, the lack of coercivity of $\mathcal{H}$ makes it impossible to directly apply the arguments of~\cite[Sect. 4.1]{C}, in particular a contraction property of the flow associated to the dynamics (see~\cite[Lemma 4.13]{C}).
However, the superposition principle and suitable optimal synthesis results will be used to characterize $m$ as the image of the initial distribution by the optimal flow associated with the Hamilton-Jacobi equation. By standard techniques for monotone operators (see Lasry and Lions~\cite{ll07}), we also obtain the uniqueness of the solution under classical assumptions.

The superposition principle has already  been used in a different approach, for instance in the articles of Cardaliaguet~\cite{C15}, Cardaliaguet, M\'esz\'aros and Santambrogio~\cite{CMS16} and Orrieri, Porretta and Savar\'e~\cite{OPS19}. In these works, the authors tackle the MFG systems of the first order using a variational approach based on two optimization problems in duality, under suitable assumptions. Then, using the superposition principle, they are able to describe the solution to the continuity equation  arising in the optimality conditions of the latter optimization problem by means of a measure on the space of continuous paths. This measure is concentrated on the set of minimizing curves for the optimal control problem underlying the Hamilton-Jacobi equation.

A similar approach to the one of the present paper was recently proposed for a class of non-coercive MFG when the generic player has some ``forbidden direction'' (see \cite{MMMT}), more precisely when, in the two dimensional case, the dynamics  is of the form: $x_1'= \alpha_1$, $x_2'=h(x_1)\alpha_2$ and  $h(x_1)$ may vanish.

In a near future, we plan to  tackle  mean field games with control on the acceleration and with  constraints (for MFGs with state constraints we refer to \cite{ABLLM, CC,CCC, CCC2}).

The paper is organized as follows. In Section \ref{Ass}, we list our assumptions, give the definition of (weak) solution to system~\eqref{eq:MFGAs} and state the existence and uniqueness results for the latter.
In Section \ref{OC},
we obtain some regularity properties for the solution $u$ of the Hamilton-Jacobi equation~\eqref{eq:MFGAs}-(i) with $m$ fixed.
 These properties, combined with the uniqueness of the optimal trajectories of the associated control problem, will be crucial for proving the main theorem.
In Section \ref{sect:c_eq}, we study the continuity equation~\eqref{eq:MFGAs}-(ii). An important ingredient is the vanishing viscosity method that is used to characterize its  solution.
Finally, Section~\ref{sect:MFG} is devoted to the proofs of the main Theorem~\ref{thm:main} on the existence of a solution and of Proposition~\ref{prp:!} on its uniqueness.
In the Appendix, following a suggestion of the referee, we establish the existence and the uniqueness of the solution to the corresponding second order MFG system as a byproduct of the estimates needed for the vanishing viscosity limit.

\section{Assumptions and main results}\label{Ass}

We consider the running cost $l(x,v,\alpha)$ of the form $
l(x,v,\alpha)=l(x,v)+\frac1 2 \vert \alpha\vert^2+\frac1 2 \vert v \vert^2$.\\
Then system~\eqref{eq:MFGA} can be written
\begin{equation}
\label{eq:MFGAs}
\left\{
\begin{array}{rll}
(i)&-\partial_t u-v\cdot D_xu+\frac1 2 \vert D_vu\vert^2-\frac1 2 \vert v \vert^2  -l(x,v)- F[m](x,v)=0,&\quad \textrm{in }\R^{2N}\times (0,T),\\
(ii)& \partial_t m +v\cdot D_xm-{\rm div}_v(D_vu\, m)=0,&\quad \textrm{in }\R^{2N}\times (0,T),\\
(iii)& m(x,v,0)=m_0(x,v),\ u(x,v,T)=G[m(T)](x,v),&\quad \textrm{on }\R^{2N},
\end{array}\right.
\end{equation}
which corresponds to $H(x,v,p_v)=\frac1 2 \vert p_v\vert^2-\frac1 2 \vert v \vert^2-l(x,v)$.

Let $\mathcal P_1$ and $\mathcal P_2$ denote the spaces of Borel probability measures on~$\re^{2N}$
 with respectively finite first and  second order moments, endowed with the Monge-Kantorovich distances~{${\bf d}_1$}, respectively {${\bf d}_2$}. 

Let $C^2(\R^{2N})$ denote the space of twice differentiable functions with continuous and bounded  derivatives up to order two. It is  endowed with the norm \\
$\|f\|_{C^{2}}:=\sup_{(x,v)\in\R^{2N}}[|f(x,v)|+|Df(x,v)|+|D^2f(x,v)|]$.

Hereafter, we shall make the following hypotheses:

\paragraph{\bf Assumptions (H)}
\begin{itemize}
\item[(H1)]\label{H1} The functions~$F$ and $G$ are real-valued continuous functions defined on $\mathcal P_1\times\re^{2N}$
\item[(H2)]\label{H2} The function~$l$ is a real-valued $C^2$ function defined on $\re^{2N}$
\item[(H3)]\label{H3} The map $m\mapsto F[m](\cdot, \cdot)$ is Lipschitz continuous from $\mathcal P_1$ to $C^{2}(\re^{2N})$; moreover, there exists~$C>0$ such that $C \ge  \|l\|_{C^2}$ and
$$\|F[m](\cdot,\cdot)\|_{C^2} + \|G[m](\cdot,\cdot)\|_{C^2}\leq C,\qquad \forall m\in \mathcal P_1
$$
\item[(H4)]\label{H4} the initial distribution~$m_0$, defined on $\re^{2N}$,  
has a compactly supported density (still named $m_0$, with a slight abuse of notation) $m_0\in C^{0,\delta}(\re^{2N})$ for some $\delta\in (0,1)$.
\end{itemize}

\begin{definition}\label{defsolmfg}
The pair $(u,m)$ is a solution of system~\eqref{eq:MFGAs} if:
\begin{itemize}
\item[1)] $u\in W_{\rm loc}^{1,\infty}(\re^{2N}\times[0,T])$, $m\in C([0,T];\mathcal P_1(\re^{2N}))$
and for all $ t\in [0,T]$, $m(t)$ is absolutely continuous with respect to Lebesgue measure on $\R^{2N}$.
Let $m(\cdot, \cdot, t)$ denote the density of $m(t)$. The function $(x,v,t)\mapsto m(x,v,t)$ is bounded.
\item[2)] equation~\eqref{eq:MFGAs}-(i) is satisfied by $u$ in the viscosity sense
\item[3)] equation~\eqref{eq:MFGAs}-(ii) is satisfied by $m$ in the sense of distributions.
\end{itemize}
\end{definition}
We can now state the main result of this paper:
\begin{theorem}\label{thm:main}
Under the assumptions $\rm{(H)}$:
\begin{enumerate}
\item System \eqref{eq:MFGAs} has a solution $(u,m)$ in the sense of Definition~\ref{defsolmfg},
\item $m$ is the image of $m_0$ by the flow
\begin{equation}\label{dyn}
\left\{
\begin{array}{ll}
x'(s)=v(s),& \quad x(0)=x, \\
v'(s)=-D_vu(x(s),v(s),s),&\quad v(0)=v.
\end{array}
\right.
\end{equation}
\end{enumerate}
\end{theorem}
\begin{proposition}\label{prp:!}
Under the additional assumptions
\begin{equation}\label{monot}
\int_{\R^{2N}}(F[m_1]-F[m_2]) d(m_1-m_2)>0\quad \textrm{and }
\int_{\R^{2N}}(G[m_1]-G[m_2]) d(m_1-m_2)\geq0  
\end{equation}
for every $m_1,m_2\in\mathcal{P}_1(\R^{2N})$, $m_1\ne m_2$, the solution found in Theorem~\ref{thm:main} is unique.
\end{proposition}
%
%
%
%

\section{The optimal control problem}\label{OC}
In this section, we tackle the optimal control problem related to equation~\eqref{eq:MFGAs}-(i) with a fixed $\overline m\in C([0,T];{\mathcal P}_1(\R^{2N}))$. To simplify the notation, we introduce the functions
\begin{equation}\label{ell}
\ell(x,v,t):=l(x,v)+F[\overline m(t)](x,v)\quad\textrm{and}\qquad g(x,v):=G[\overline m(T)](x,v),
\end{equation}
which, from the set assumptions $\rm{(H)}$, satisfy
\begin{equation}\label{HOC}
\|\ell(\cdot,\cdot,t)\|_{C^2},\, \|\ell(x,v,\cdot)\|_{C}, \, \|g\|_{C^2}\leq C\qquad \forall t\in[0,T], (x,v)\in\R^{2N}.
\end{equation}
With the new notation,  the optimal control  problem to be solved by a representative agent whose state at time $t$ is $(x,v)$ is to find the control law $\alpha$ in order to minimize
\begin{equation}\label{def:OC}
J_t(\xi,\eta, \alpha) =\displaystyle\int_t^T
\left[\frac{|\alpha|^2}{2}+\frac{|\eta|^2}{2}+\ell(\xi(s), \eta(s), s)\right]ds+g(\xi(T),\eta(T)),
\end{equation}
by following the trajectory~\eqref{eq:HJA}. Then the Cauchy problem given by ~\eqref{eq:MFGAs}-(i) and its terminal condition becomes
\begin{equation}\label{HJ}
\left\{\begin{array}{ll}
-\partial_t u-v\cdot D_xu+\frac1 2 \vert D_vu\vert^2-\frac1 2 \vert v \vert^2  -\ell(x,v,t)=0&  \textrm{in }\R^{2N}\times (0,T),\\
u(x,v,T)=g(x,v)& \textrm{on }\R^{2N}.
\end{array}\right.
\end{equation}
From~\eqref{def:OC}, it is obvious that the control~$\alpha$ must be chosen in~$L^2(t,T;\R^{N})$. Therefore, we can introduce the value function as follows:
\begin{definition} The value function for the cost $J_t$ defined in \eqref{def:OC} with dynamics~\eqref{eq:HJA} is
\begin{equation}\label{repr}u(x,v,t):=\inf\left\{ J_t(\xi,\eta, \alpha):\, (\xi,\eta, \alpha)\in \mathcal A(x,v,t)\right\}
\end{equation}
where
\begin{equation}
\label{eq:constraint}
{\mathcal A}(x,v,t)\!=\!\left\{ (  \xi, \eta, \alpha):\;\left|
  \begin{array}[c]{l} (\xi,\eta)\in AC([t,T]; \R^{2N}),\;\alpha\in L^2(t,T; \R^N),\\
 (\xi,\eta,\alpha) \textrm{ satisfy } \eqref{eq:HJA} \textrm{ and }\xi(t)= x, \eta(t)=v
  \end{array}\right.\right\}.    
\end{equation}
\end{definition}

\begin{lemma}\label{DPP} 
\begin{itemize}
\item[i)] {\it (Existence of an optimal control.)} For every $(x,v,t)\in\R^N\times\R^N\times (0,T)$, there exists an optimal control $\alpha^*$ for $u(x,v,t)$.
\item[ii)] {\it (Concatenation.)} Let $(\xi^*,\eta^*)$ be an optimal trajectory for $u(x,v,t)$ corresponding to the control law $\alpha^*$. For $r\in(t,T)$, let $(\tilde\xi^*,\tilde\eta^*)$ be an optimal trajectory for $u(\xi^*(r),\eta^*(r),r)$ with control $\tilde \alpha^*$. Then the concatenation of $\alpha^*$ and $\tilde\alpha^*$ at time $r$ is optimal for $u(x,v,t)$ and, moreover, 
\[
u(x,v,t)=u(\xi^*(r),\eta^*(r),r)+\int_t^r
\left[\frac{|\alpha^*|^2}{2}+\frac{|\eta^*|^2}{2}+\ell(\xi^*(s), \eta^*(s), s)\right]ds.
\]
\item[iii)] Under the same assumption as in point (ii), the control $\alpha^*_{\mid [r,T]}$ is optimal for \\ $u(\xi^*(r),\eta^*(r),r)$.
\item[iv)] {\it(Dynamic Programming Principle.)} The Dynamic Programming Principle holds, namely
\begin{equation*}
u(x,v,t)=\min_{(\xi,\eta,\alpha)\in{\mathcal A}(x,v,t)}\left\{
u(\xi(r),\eta(r),r)+\int_t^r\frac{|\alpha(s)|^2}{2}+\frac{|\eta(s)|^2}{2} + \ell(\xi(s),\eta(s),s)\, ds
\right\}.
\end{equation*}
\end{itemize}
\end{lemma}
\begin{proof}
{\sl(i)}: let $\{\alpha_n\}_n$ be a sequence of minimizing control laws  and $(\xi_n, \eta_n)$ be the solution of (\ref{eq:HJA}) corresponding to $\alpha_n$. Then, the boundedness of $\ell$ and the definition of $J_t$ ensure that $\|\alpha_n\|_{L^2(t,T;\R^N)}$ are uniformly bounded. Then, possibly after extracting a subsequence, $\alpha_n\rightharpoonup \alpha^*$ in $L^2(t,T;\R^N)$, $\eta_n\to \eta^*$ in $C([t,T];\R^N)$ and $\xi_n\to \xi^*$ in  $C^1([t,T];\R^N)$. The lower semi-continuity of $J_t$ yields that $\alpha^*$ is optimal.

Points {\sl(ii)}, {\sl(iii)} and {\sl(iv)} are obtained by arguing exactly as in \cite[Proposition 5.1]{MMMT} (points (1), (2) and (4) respectively), see also \cite{C}.
\end{proof}

\begin{lemma}\label{L1}
The value function~$u$ has the following properties:
\begin{enumerate}
\item (Lipschitz continuity in $x$ and local Lipschitz continuity in $v$) there exists a positive constant $C$, depending only on the constants in assumptions $\rm{(H)}$, such that
\begin{eqnarray*}
|u(x,v,t)-u(x',v,t)|&\leq& C|x-x'|\\ 
|u(x,v,t)-u(x,v',t)|&\leq& C(1+|v|+|v'|)|v-v'|
\end{eqnarray*}
for every $x,x',v,v'\in\R^N$, $t\in[0,T]$.
\item (Local Lipschitz continuity in $t$) there exists a positive constant $C$, depending only on the constants in assumptions $\rm{(H)}$, such that
\begin{eqnarray*}
|u(x,v,t)-u(x,v,t')|&\leq& C(1+|v|^2)|t-t'| \qquad \forall x,v\in\R^N,\,t,t'\in[0,T].
\end{eqnarray*}
\end{enumerate}
\end{lemma}

\begin{proof}  
  \begin{enumerate}
  \item Fix $t\in [0,T)$. 
Let $\alpha$ be an optimal control law for $u(x,v, t)$ i.e.,
\begin{equation}
\label{eq:HJ31}
u(x, v, t) = \int_t^T\frac12 |\alpha(s)|^2+\frac12 |v(s)|^2+\ell(x(s),v(s),s)\,ds+g(x(T),v(T)),
\end{equation}
where~$(x(\cdot),v(\cdot))$ obeys to the dynamics \eqref{eq:HJA}.

We consider the path $(y(\cdot), w(\cdot))$ starting from $(y, w)$, with control $\alpha(\cdot)$.
Hence, we obtain
\begin{eqnarray*}
y(s)&=&y+w(s-t)+\int_t^s\int_t^{\theta}\alpha(\tau) \,d\tau d\theta=y-x+x(s)+(w-v)(s-t),\\
w(s)&=&w+\int_t^s \alpha(\tau)\,d\tau=w-v+v(s).
\end{eqnarray*}
Note that 
\begin{equation}\label{diff}
v(s)-w(s)=v-w,\  x(s)-y(s)=x-y+(v-w)(s-t).
\end{equation}
The definition of the value function \eqref{repr} and relation \eqref{eq:HJ31} imply
\begin{eqnarray*}
u(y, w, t)&\leq& \int_t^T\frac12 |\alpha(s)|^2+\frac12 |w(s)|^2+
\ell(y(s),w(s), s)\,ds+g(y(T), w(T))\\
&\leq& u(x, v, t) -\int_t^T\frac12 |v(s)|^2-\ell(x(s),v(s), s)\,ds-g(x(T),v(T))\\
&&\quad +\int_t^T\frac12 |w(s)|^2+\ell(y(s),w(s), s)\,ds+g(y(T), w(T))\\
&\leq& 
u(x, v, t)+ \int_t^T L_{\ell}(|x(s)-y(s)|+|v(s)-w(s)|)\, ds \\
&&\quad+L_g(|x(T)-y(T)|+|v(T)-w(T)|)+ \int_t^T\frac12 (|w(s)|^2-|v(s)|^2)ds,
\end{eqnarray*}
where  $L_\ell$ and $L_g$ denote respectively the Lipschitz constants of~$\ell$ and $g$ w.r.t. $(x,v)$.
Hence, by \eqref{diff},
\begin{eqnarray*}
&&\int_t^T\frac12 (|w(s)|^2-|v(s)|^2)ds= \int_t^T\frac12 |w-v|\cdot|w(s)+v(s)|ds\leq\\
&&|w-v|\int_t^T|w+v+2\int_t^s\alpha(\tau)d\tau|ds\leq 
C|w-v|(|w|+|v|+ 1),
\end{eqnarray*}
where the last inequality comes from \eqref{stimaL8} of Corollary \ref{coro:regularity}  below.
Hence we obtain
\begin{eqnarray}\label{eq:9}
&&u(y, w, t)\leq  
u(x, v, t)+ C|x-y|+K(v,w)|v-w|,
\end{eqnarray}
where $K(v,w)=C(|w|+|v|+ 1)$. 
Reverting the roles of $(x,v)$ and $(y,w)$, we get the first result.

\medskip

\item
We fix $(x,v)$.
From the concatenation property of optimal trajectories established in Lemma \ref{DPP}, if $\alpha$ is optimal for $u(x,v,t)$ and $(x(s), v(s))$ 
is the associated optimal trajectory, then
$$u(x,v,t)=u(x(s), v(s), s)+\int_t^s\frac12 |\alpha(r)|^2+\frac12 |v(r)|^2+\ell(x(r),v(r),r)\,dr$$
for any $s\in[t,T]$.
Then
\begin{eqnarray*}
&&|u(x,v,t)-u(x,v,s)|\leq |u(x,v,t)-u(x(s),v(s), s)|+
|u(x(s), v(s), s)-u(x,v,s)|\\
&&\leq 
\int_t^s\frac12 |\alpha(r)|^2+\frac12 |v(r)|^2+|\ell(x(r),v(r),r)|\,dr+ L|x(s)-x|+L(v)|v(s)-v|,
\end{eqnarray*}
where the last two terms come from the Lipschitz continuity of $u$ w.r.t. $(x,v)$: $L$ is the Lipschitz constant of $u$ with respect to $x$ and $L(v)$ is 
a local Lipschitz constant of $u$ with respect to $v$.

From \eqref{eq:HJA} and the bound~\eqref{stimaL8} in Corollary \ref{coro:regularity}  below, 
we get the bounds for $x(s)$, $v(s)$ and $\alpha$, i.e.
$|v(s)-v|+|x(s)-x|\leq C(1+|v|)|s-t|$,  hence 
\begin{equation}\label{L2}
|u(x,v,t)-u(x,v,s)|\leq C(1+|v|^2)|s-t|,
\end{equation}
which ends the proof. 
\end{enumerate}
\end{proof}
\begin{proposition}\label{exuniq}
The value function  defined in \eqref{repr} is the unique viscosity solution to~\eqref{HJ} with an at most quadratic growth in $(x,v)$. Moreover, there exists a positive constant~$C$ such that
\begin{equation}\label{eq:stimau}
-C \leq u(x,v,t)\leq C(1+\vert v\vert^2) \qquad \forall(x,v,t)\in\R^N\times \R^N\times[0,T].
\end{equation}
\end{proposition}
\begin{proof}
Let us first establish that the value function fulfills \eqref{eq:stimau} and solves~\eqref{HJ} in the viscosity sense. Actually, taking $\alpha \equiv 0$ in \eqref{eq:HJA}, we get $\eta(s)=v$ and $\xi(s)=x+v(s-t)$; then, thanks to the boundedness of $\ell$ in~\eqref{HOC}, the value function verifies~\eqref{eq:stimau}.
Moreover, by Lemma~\ref{L1}, it is also continuous; hence, using the DPP in Lemma~\ref{DPP}-(iv), it is also a solution to \eqref{HJ}.\\
The uniqueness part of the statement is an immediate consequence of the comparison principle stated in \cite[Theorem 2.1]{DLL}.
\end{proof}

The following lemma deals with the semi-concavity of $u(x,v, t)$ w.r.t. $(x,v)$:
\begin{lemma}\label{semi-concav}
Under Hypothesis $\rm{(H)}$, $u(x,v,t)$ is semi-concave w.r.t. $(x,v)$ with a linear modulus of semi-concavity, which depends only 
 on the constants in assumptions $\rm{(H)}$.
\end{lemma}
\begin{proof}
For any $(x,v)$, $(y, w)$ and $\lambda\in[0,1]$, 
consider $x_{\lambda}:=\lambda x+(1-\lambda)y$, $v_{\lambda}:=\lambda v+(1-\lambda)w$.

Let $\alpha$ be an optimal control for~$u(x_{\lambda}, v_{\lambda}, t)$; hence, the associated trajectory is
\begin{equation}\label{icslanda}
x_{\lambda}(s)=x_{\lambda}+ v_{\lambda}(s-t) +\int_t^s\int_t^{\theta}\alpha(\tau) \,d\tau d\theta,\ 
v_{\lambda}(s)= v_{\lambda}+ \int_t^s \alpha(\tau)\,d\tau
\end{equation}
and 
$$u(x_{\lambda}, v_{\lambda}, t)=\int_t^T\frac12 |\alpha(s)|^2+\frac12 |v_{\lambda}(s)|^2+
\ell(x_{\lambda}(s),v_{\lambda}(s),s)ds+ g(x_{\lambda}(T), v_{\lambda}(T)).$$

Let $(x(s), v(s))$ be the trajectory starting at $(x,v)$ at time $t$ with control $\alpha$ and $(y(s), w(s))$ the trajectory starting at $(y,w)$ at time $t$ still with control $\alpha$.

We have to estimate 
\begin{displaymath}
  \begin{split}
    &\lambda u(x,v, t) +(1-\lambda)u(y,w, t)-u(x_{\lambda}, v_{\lambda}, t)\\
    \le &\int_t^T\frac12\lambda |v(s)|^2+ (1-\lambda)\frac12 |w(s)|^2-\frac12 |v_{\lambda}(s)|^2 ds\\
    &+\int_t^T\lambda \ell(x(s),v(s),s)+(1-\lambda) \ell(y(s),w(s),s)- \ell(x_{\lambda}(s),v_{\lambda}(s),s)ds\\
    &+\lambda g(x(T), v(T))+(1-\lambda)g(y(T), w(T))-g(x_{\lambda}(T), v_{\lambda}(T)).
  \end{split}
\end{displaymath}
Since 
\begin{equation}\label{dintutte}
v(s)=v+ \int_t^s \alpha(\tau)\,d\tau,\  w(s)=w+ \int_t^s \alpha(\tau)\,d\tau,\  v_{\lambda}(s)=\lambda v+(1-\lambda)w +\int_t^s \alpha(\tau)\,d\tau,
\end{equation}
we get
\begin{equation}
  \begin{split}
    &\lambda \frac12|v(s)|^2+ (1-\lambda)\frac12 |w(s)|^2-\frac12 |v_{\lambda}(s)|^2\label{eqsemi} \\
=&(\lambda v+(1-\lambda)w-\lambda v-(1-\lambda)w)\int_t^s \alpha(\tau)\,d\tau+ \lambda \frac{|v|^2}{2}+ (1-\lambda)\frac{|w|^2}{2}-\frac12|\lambda v+ (1-\lambda)w|^2\\
=&
\frac12\lambda(1-\lambda)|v|^2+ \frac12\lambda(1-\lambda)|w|^2-\lambda(1-\lambda)v\cdot w= 
\frac12\lambda(1-\lambda)|v-w|^2.
  \end{split}
\end{equation}
Hence
\begin{equation}\label{unouno}
  \int_t^T\frac12\lambda |v(s)|^2+ (1-\lambda)\frac12 |w(s)|^2-\frac12 |v_{\lambda}(s)|^2 ds=\frac12\lambda(1-\lambda)|v-w|^2(T-t).
\end{equation}
Now, we have to estimate the terms $\lambda \ell(x(s), v(s), s) +(1-\lambda)\ell(y(s), w(s), s)- \ell(x_{\lambda}(s),v_{\lambda}(s),s)$ and 
$\lambda g(x(T), v(T))+(1-\lambda)g(y(T), w(T))-g(x_{\lambda}(T), v_{\lambda}(T))$.
We write the algebra for the second term, since the treatment of the first term is similar.
The Taylor expansion of $g$ centered at  $(x_{\lambda}(T), v_{\lambda}(T))$ gives
\begin{equation}
\label{taypro}
g(x(T),v(T))= g(x_{\lambda}(T), v_{\lambda}(T))+  Dg(x_{\lambda}(T), v_{\lambda}(T))(x(T)-x_{\lambda}(T), v(T)-v_{\lambda}(T))+ R_1,
\end{equation}
where $R_1$ is the remaining term in the expansion, namely
\begin{equation}
  \label{R1}
R_1=\frac 1 2 (x(T)-x_{\lambda}(T), v(T)-v_{\lambda}(T))D^2g(\xi_1,\eta_1)(x(T)-x_{\lambda}(T), v(T)-v_{\lambda}(T))^T,
\end{equation}
for suitable $\xi_1, \eta_1$.\\
From \eqref{icslanda} and \eqref{dintutte}, we get
\begin{equation}
  \label{reltutte}
  \begin{array}[c]{rcl}
    x(s)-x_{\lambda}(s)&=& (1-\lambda)((x-y)+(v-w)(s-t)),\\
 v(s)-v_{\lambda}(s) &=& (1-\lambda)(v-w),\\
y(s)-x_{\lambda}(s)&=& \lambda((y-x)+(w-v)(s-t)),\\
 w(s)-v_{\lambda}(s) &=& \lambda(w-v),
  \end{array}
\end{equation}
hence the error term can be written as 
\begin{equation}\label{R12}
R_1= \frac 1 2(1-\lambda)^2(x-y+(v-w)(T-t), v-w)D^2g(\xi_1,\eta_1)(x-y+(v-w)(T-t), v-w)^T.
\end{equation}
Similarly
\begin{displaymath}
  g(y(T), w(T))=
g(x_{\lambda}(T), v_{\lambda}(T))+  Dg(x_{\lambda}(T), v_{\lambda}(T))(y(T)-x_{\lambda}(T), w(T)-v_{\lambda}(T))+ R_2,
\end{displaymath}
where
\begin{displaymath}
  \begin{split}
   R_2=
& \frac 1 2 (y(T)-x_{\lambda}(T), w(T)-v_{\lambda}(T))D^2g(\xi_2,\eta_2)(y(T)-x_{\lambda}(T), w(T)-v_{\lambda}(T))^T\\
=&\frac 1 2 \lambda^2 (y-x+(w-v)(T-t), w-v)D^2g(\xi_2,\eta_2)(y-x+(w-v)(T-t), w-v)^T.
  \end{split}
\end{displaymath}
At this point, taking into account that
 from \eqref{reltutte},
\begin{equation}
  \label{zeroDg}
  \begin{split}
&\lambda Dg(x_{\lambda}(T), v_{\lambda}(T))(x(T)-x_{\lambda}(T), v(T)-v_{\lambda}(T))\\
&+(1-\lambda) Dg(x_{\lambda}(T), v_{\lambda}(T))(y(T)-x_{\lambda}(T), w(T)-v_{\lambda}(T))\\
=& Dg(x_{\lambda}(T), v_{\lambda}(T))(\lambda(x(T)-x_{\lambda}(T))\\ &\ds +(1-\lambda)(y(T)-x_{\lambda}(T)),
\lambda(v(T)-v_{\lambda}(T))+(1-\lambda)(w(T)-v_{\lambda}(T)))\\
=&0,
\end{split}
\end{equation}
we obtain that 
\begin{equation}\label{gg}
  \begin{array}[c]{ll}
&\ds  \lambda g(x(T), v(T))+(1-\lambda)g(y(T), w(T))-g(x_{\lambda}(T), v_{\lambda}(T))\\ =&     \lambda R_1+
(1-\lambda) R_2\\
\le &\ds (1-\lambda)\lambda C_T\|D^2g\|_{\infty}(|x-y|^2+|v-w|^2).
  \end{array}
\end{equation}
Hence from \eqref{unouno},  \eqref{zeroDg}, \eqref{gg} we get
\begin{displaymath}
  \begin{split}
    &\lambda u(x,v, t) +(1-\lambda)u(y,w, t)-u(x_{\lambda}, v_{\lambda}, t)\\
\le &\frac{\lambda(1-\lambda)}2 |v-w|^2(T-t)+
C_T (1-\lambda)\lambda  \left(\|D^2g\|_{\infty} +\|D^2\ell\|_{\infty}\right) \left(|x-y|^2+|v-w|^2\right).
  \end{split}
\end{displaymath}
We obtain that $u$ is semi-concave in $(x,v)$ with a linear modulus of semi-concavity.
\end{proof}
Pontryagin's maximum principle yields the following necessary optimality conditions:
\begin{proposition}[Necessary conditions for optimality]\label{prop:pontriagin}
\label{MPP}
Let $(x^*, v^*, \alpha^*)$ be optimal for~$u(x,v,t)$ in~\eqref{repr}. There exists an arc $p=(p_x,p_v)\in AC([t,T];\R^N\times\R^N)$, hereafter called the costate,  such that
\begin{enumerate}
\item $(\alpha^*, x^*, v^*, p)$ satisfies
the {\it adjoint equations}: for a.e. $s\in[t,T]$,
\begin{eqnarray}
&&p_x' =D_x\ell( x^*, v^*, s),\label{tag:adjoint1}\\
&&p_v'=-p_x+v^*+D_v\ell( x^*, v^*, s),\label{tag:adjoint2}
\end{eqnarray}
the {\it transversality condition}
\begin{equation}\label{tag:transversality}
p(T)=-D g(x^*(T), v^*(T)),
\end{equation}
together with the {\it maximum condition}: for almost all $s\in [t,T]$,
\begin{multline}\label{tag:max}
\max_{\alpha}p_x\cdot v^*+p_v\cdot\alpha-\dfrac{|\alpha|^2}2-\dfrac{|v^*|^2}2=
p_x\cdot v^*+p_v\cdot\alpha^*-\dfrac{|\alpha^*|^2}2-\dfrac{|v^*|^2}2.\end{multline}
\item The optimal control $\alpha^*$ is given by
\begin{equation}
\alpha^*=p_v, \text{  a.e in }[t,T].\label{tag:alpha*}
\end{equation}
\item The triple~$(x^*, v^*, p)$ satisfies the system of differential equations: for a.e. $s\in[t,T]$
\begin{eqnarray}
&&x'=  v,\label{tag:1} \\
&&v'= p_v, \label{tag:2}\\
&&p_x'= D_x\ell(x,v, s),\label{tag:3}\\
&&p_v'=-p_x+v+D_v\ell(x,v, s),\label{tag:4}
\end{eqnarray}
with the mixed boundary conditions $x^*(t)= x$, $v^*(t)= v$, $p(T)=-D g(x^*(T), v^*(T))$.
\end{enumerate}
\end{proposition}

\begin{proof} 1. Hypothesis~\eqref{HOC} ensures that our control problem satisfies the assumption~\cite[Hypothesis 22.16]{Cla},
 so we can invoke~\cite[Theorem 22.17]{Cla} on the maximum principle for problems with unbounded control.
Moreover, since there is no constraint on the state variable at $T$, the same arguments as in ~\cite[Corollary 22.3]{Cla}
 ensure that the necessary conditions hold in normal form.

2. The maximum condition \eqref{tag:max} implies that
\[D_{\alpha}\left(p_x\cdot v^*+p_v\cdot \alpha-\dfrac{|\alpha|^2}2-\dfrac{|v^*|^2}2-f(x^*, v^*)\right)_{\alpha=\alpha^*}=0\quad \text{for a.e. }s\in [t, T]\]
from which we get \eqref{tag:alpha*}.

3. Conditions \eqref{tag:1} -- \eqref{tag:2} follow directly from \eqref{eq:HJA} and \eqref{tag:alpha*}. Conditions \eqref{tag:3} and \eqref{tag:4} coincide with \eqref{tag:adjoint1}, \eqref{tag:adjoint2}.
\end{proof}

\begin{corollary}[Feedback control and regularity]\label{coro:regularity}
Let $(x^*, v^*, \alpha^*)$ be optimal for $u(x, v, t)$ and $p=(p_x,p_v)$ be the  related costate as in Proposition~\ref{prop:pontriagin}. Then:
\begin{enumerate}
\item The costate $p$ is uniquely expressed in terms of $x^*, v^*$ for every $s\in [t, T]$ by
\begin{equation}
\!\begin{cases}\label{tag:p}
p_x(s)\!\!&\!\!\!=-D_xg(x^*(T), v^*(T))-\!\!\displaystyle\int_s^T \!\!D_x\ell(x^*(\tau), v^*(\tau),\tau)\,d\tau,\\
p_v(s)\!\!&\!\!\!=-D_{v}g(x^*(T), v^*(T))-\displaystyle\int_s^T D_{v}\ell(x^*(\tau),v^*(\tau), \tau)+v^*(\tau)-p_x(\tau)\,d\tau.\\
\end{cases}
\end{equation}
\item The optimal control
 $\alpha^*$ is a feedback control {\rm (}i.e., a function of $x^*, v^*${\rm )}, uniquely expressed   in terms of $x^*, v^*$ for a.e. $s\in [t, T]$ by
\begin{equation}\label{tag:alpha}
\alpha^*(s)=p_v(s).
\end{equation}
\item  The optimal trajectory $(x^*,v^*)$ and the optimal control $\alpha^*$ are of class $C^1$. 
In particular the equalities \eqref{tag:alpha*} -- \eqref{tag:alpha} do hold for every $s\in [t, T]$.
Moreover
\begin{equation}\label{stimaL8}
  \begin{array}[c]{rcl}
    \|v^*\|_{C^1}+\|\alpha^*\|_{C^1}&\leq& C(1+|v|),\\
\|x^*\|_{C^1}&\leq& |x|+C(1+|v|).
  \end{array}
\end{equation}
\item Assume that,  for some $k\in\mathbb N$, $D_x\ell(x,v, s)$, $D_v\ell(x,v, s)$ are of class $C^k$.
Then $(x^*, v^*)$, $p$ and $\alpha^*$  are of class   $C^{k+1}$.
\end{enumerate}
\end{corollary}
\begin{proof}Point~$1$ is obtained integrating \eqref{tag:3}--\eqref{tag:4} and taking into account the final time condition $p(T)=-D g(x^*(T),v^*(T))$.
Point~$2$ follows from \eqref{tag:alpha*}.

Proof of point~$3$. Since $x^*, v^*$ are continuous by the definition of admissible trajectories in \eqref{eq:constraint}, the continuity of  $\alpha^*$ follows from \eqref{tag:p} and~\eqref{tag:alpha}. Then (\ref{eq:HJA}) implies $v^*\in C^1$ and also $x^*\in C^1$.
Relations~\eqref{tag:p},~\eqref{tag:alpha} (and the regularity of $\ell$) imply, respectively,  that  $p$ and $\alpha^*$  are of class $C^1$. By \eqref{tag:2}, we get that $v^*$ is $C^2$. Let us now prove the bounds \eqref{stimaL8}. To this end, we observe that equations \eqref{tag:2} and \eqref{tag:4} entail
\begin{equation*}
(v^*)''(\tau)-v^*(\tau)=-p_x(\tau)+D_v\ell(x^*(\tau),v^*(\tau),\tau)
\end{equation*}
where, by \eqref{tag:p} and ($H2$), the right hand side is bounded uniformly in $x$ and $v$, $\tau\in [t,T]$. Moreover,
\begin{itemize}
\item $(v^*)(t)=v$
\item by \eqref{tag:2}, \eqref{tag:transversality} and the regularity of $g$, $(v^*)'(T)$ is bounded uniformly in $x$ and $v$. 
\end{itemize}
Hence, using the method of variation of constants for the above ordinary differential equation with assigned the values of $v^*(t)$ and of $(v^*)'(T)$, we get the estimate for $v^*$ and for $\alpha^*$. Integrating $v^*$, we get the estimate for $x^*$.

Proof of point~$4$. The relations~\eqref{tag:p} and the $C^1$-regularity of $x^*, v^*$ and $p$  imply that, actually,  $p\in C^2$. Therefore, \eqref{tag:alpha} gives the $C^2$-regularity of $\alpha^*$ and, finally, \eqref{dyn} yields the $C^2$-regularity of $x^*, v^*$.
Further regularity of $x^*, v^*$, $\alpha^*$ and $p$ follows by a standard bootstrap inductive argument.
\end{proof}

\begin{remark}\label{contrcont}
Taking advantage of Corollary~\ref{coro:regularity}-(3), we will always consider the representation of the optimal control~$\alpha^*$ which belongs to~$C^1$.
\end{remark}

 Corollary~\ref{th:nobifurc} that follows implies that the optimal trajectories for $u(x,v, t)$ do not bifurcate at any time $r>t$.
\begin{corollary}\label{th:nobifurc}
Under Hypothesis~\eqref{HOC}, let $(x^*, v^*)$ be an  optimal trajectory for $u(x,v, t)$.
For every $t< r< T$, there are no other optimal trajectories for $u(x^*(r), v^*(r), r)$ other than $(x^*, v^*)$ restricted to $[r,T]$.
\end{corollary}

\begin{proof} 1.
Let $r\in (t, T)$ and $(y^*, w^*)$ be an optimal  trajectory for $u(x^*(r), v^*(r), r)$.
Lemma~\ref{DPP} ensures that $(z^*, \nu^*)$, the concatenation  of $(x^*, v^*)$ with $(y^*, w^*)$ at $r$ is an optimal trajectory 
for $u(x,v, t)$. Let $p:=(p_x, p_v), q:=(q_x, q_v)$ be the costates corresponding respectively to  $(x^*, v^*)$ and to $(z^*, \nu^*)$.
Both $(x^*, v^*, p)$ and $(z^*, \nu^*, q)$ satisfy \eqref{tag:1} -- \eqref{tag:4} on $[t, T]$.
Now, Corollary~\ref{coro:regularity} shows that $(x^*, v^*)$ and $(z^*, \nu^*)$ are of class $C^1$.
Since $x^*=z^*$,  $v^*=\nu^*$ on $[t, r]$, we choose $\tau$ such that $t<\tau<r$.
From \eqref{tag:2}, we get  \[p_v(\tau)=q_v(\tau).\]
Moreover, from \eqref{tag:2} and \eqref{tag:4}, we also get that
\[p_x(\tau)=q_x(\tau).\]
Therefore, both $(x^*, v^*, p)$ and $(z^*, \nu^*, q)$ are solutions to the same Cauchy problem on $[t, T]$ 
with the first order differential system \eqref{tag:1}-\eqref{tag:4} and Cauchy data at $\tau$.
 The regularity assumptions on $\ell, g$ and Cauchy-Lipschitz Theorem guarantee the uniqueness of the solution. 
Thus $x^*=z^*$, $v^*=\nu^*$ on $[\tau,T]$, from which we obtain the desired identities
$x^*=y^*$ and  $v^*=w^*$ on $[r, T]$.
\end{proof}

\begin{definition}\label{def:cal_U}
For any~$(x,v,t)\in\R^{2N}\times[0,T]$, let ${\mathcal U}(x,v,t)$ denote the set of optimal controls for the value function~$u(x,v,t)$ defined in~\eqref{repr}.
\end{definition}
\begin{remark}\label{contrcont2}
Lemma~\ref{DPP}-(i) and Remark~\ref{contrcont} ensure that $\emptyset\ne {\mathcal U}(x,v,t)\subset C^1([t,T];\R^N)$.
\end{remark}

\begin{lemma}\label{4.9} The following properties hold:
\begin{enumerate}
\item The function $u(x,\cdot,t)$ is differentiable at $v$
  if and only if the set $\{\alpha(t):\, \alpha\in {\mathcal{U}}(x,v, t)\}$ is a singleton.
Moreover $D_vu(x,v, t)=-\alpha(t)$.
\item In particular, if $\mathcal U(x,v, t)$ is a singleton, then, calling $ (x(s),v(s))$ 
the optimal trajectory associated to the singleton  $\mathcal U(x,v, t)$,
   $D_vu(x(s),v(s), s)$ exists for any $s\in [t,T]$.
\item If $  u(\cdot,\cdot,t) $ is differentiable at $(x,v)$, then $\mathcal U(x,v, t)$ is a singleton.
\end{enumerate}\end{lemma}
\begin{proof}
1. We prove that  if $D_vu(x,v,t)$ exists, then all $\alpha(\cdot)\in {\mathcal{U}}(x,v, t)$ take the same value  $\alpha(t)$ at $t$ 
 and $D_vu(x,v,t)=-\alpha(t)$.
If  $\alpha(\cdot)\in {\mathcal{U}}(x,v, t)$, calling $(x(\cdot), v(\cdot))$ the corresponding optimal trajectory, then
$$u(x,v,t)=
 \int_t^T\frac12 |\alpha(s)|^2+\frac12 |v(s)|^2+
\ell(x(s),v(s), s)\,ds+g(x(T), v(T)),$$
and
$(x(\cdot), v(\cdot))$ and $\alpha(\cdot)$ satisfy the necessary conditions for optimality proved in Proposition
\ref{prop:pontriagin}.
Take $h=(h_1,h_2)\in\R^{2N}$ and consider the solution $(y(\cdot), w(\cdot))$ of (\ref{eq:HJA}) with initial condition $(y(t),w(t))=(x+h_1, v+h_2)$ and control~$\alpha$, namely
\begin{eqnarray*}
y(s)&=&x+h_1+(v+h_2)(s-t)+\int_t^s\,\int_t^{\theta}\alpha(\tau)d\tau d\theta=x(s)+h_1+h_2(s-t),\\
w(s)&=&v+h_2+\int_t^s\,\alpha(\tau)d\tau=v(s)+h_2.
\end{eqnarray*}
Hence,
\begin{equation}
  \label{eq:1}
  \begin{array}[c]{ll}
    & u(x+h_1,v+h_2,t)-u(x,v,t)\\
\le & \ds \int_t^T\frac12 |w(s)|^2-\frac12 |v(s)|^2+
\ell(y(s),w(s), s)-\ell(x(s),v(s), s)\,ds\\ &+g(y(T), w(T))-g(x(T), v(T))\\
=& \ds \int_t^T\frac12 |v(s)+h_2|^2-\frac12 |v(s)|^2+
\ell(x(s)+h_1+h_2(s-t),v(s)+h_2, s)-\ell(x(s),v(s), s)\,ds \\ 
&\ds+
g(x(T)+h_1+h_2(T-t), v(T)+h_2)-g(x(T), v(T))\\
=& \ds \int_t^T\frac12 h_2^2+h_2\cdot v(s)+
\ell(x(s)+h_1+h_2(s-t),v(s)+h_2, s)-\ell(x(s),v(s), s)\,ds\\ 
&\ds+
g(x(T)+h_1+h_2(T-t), v(T)+h_2)-g(x(T), v(T)).
  \end{array}
\end{equation}
The arbitrariness of the sign of the components of $(h_1,h_2)$ and the differentiability of $u$ w.r.t. $v$ yields
\begin{displaymath}
  \begin{array}[c]{ll}
    D_vu(x, v, t)=&\ds \int_t^Tv(s)ds+\int_t^T
    D_x\ell(x(s),v(s), s)(s-t)+D_v\ell(x(s),v(s),s)\,ds\\
    &+D_xg(x(T), v(T))(T-t)+D_vg(x(T), v(T)). 
  \end{array}
\end{displaymath}
By \eqref{tag:3} and~\eqref{tag:transversality}, we obtain
\begin{displaymath}
  \begin{array}[c]{rcl}
\ds \int_t^T     D_x\ell(x(s),v(s),s)(s-t)ds&=& \ds \int_t^Tp_x'(s)(s-t)ds= p_x(T)(T-t)- \int_t^Tp_x(s)ds\\ &= & \ds
 -D_xg(x(T), v(T))(T-t)-\int_t^Tp_x(s)ds.
  \end{array}
\end{displaymath}
Hence
\begin{displaymath}
  \begin{array}[c]{rcl}
    D_vu(x,v,t)&=&\ds \int_t^T(v(s)-p_x(s)+D_v\ell(x(s),v(s),s))ds+ D_vg(x(T), v(T))\\ &=& \ds
\int_t^Tp_v'(s)ds+D_vg(x(T), v(T))\\ &=&-p_v(t) =-\alpha(t),
  \end{array}
\end{displaymath}
where the last two equalities are due to~\eqref{tag:4},\eqref{tag:alpha*} and the terminal condition for $p$.
This uniquely determines the value of $\alpha(\cdot)$ at time~$t$.\\

Conversely we prove that, if all $\alpha(\cdot)\in {\mathcal{U}}(x,v,t)$ take the same value
$\alpha(t)$ at $t$, then  $D_vu(x,v, t)$ exists. Fix $x$ and $t$.
 From the semi-concavity of $u(x,\cdot,t)$, the differentiability of   $u(x, \cdot, t)$ at $v$ will follow from the fact that $D_v^*u(x,v,t)$ is a singleton (see \cite[Proposition 3.3.4]{CS}).
Recall that  the set of reachable gradients of $u(x,\cdot,t)$ is defined by
\begin{displaymath}
D_v^*u(x,v,t)=\left\{
\chi\in \R^{N}:  \exists (v_n)_{n\in \N}  \hbox{ with} \left| \begin{array}[c]{ll} & \ds \lim_{n\to\infty} v_n= v, \\
 &u(x,\cdot, t) \hbox{ is differentiable at } v_n ,    
\\ & \ds \lim_{n\to \infty} D_vu(x,v_n,t)  =\chi.
  \end{array}\right.
   \right \}  .
\end{displaymath}
Take $\chi\in D_v^*u(x,v,t)$. By definition of $D_v^*u(x,v,t)$ there exist  sequences 
$\{v_n\}$, $\{\chi_n=D_vu(x, v_n, t)\}$ such that
\begin{equation}\label{1DGv}
v_n\to v\quad\hbox{and}\quad \chi_n\to \chi.
\end{equation}
Consider $\alpha_n\in\mathcal U(x, v_n, t)$;  by the other part of the statement (already proven), we know that
\begin{equation}\label{2DG}
-\alpha_n(t)=D_vu(x,v_n, t)=\chi_n.
\end{equation}
From estimate \eqref{stimaL8} in Corollary \ref{coro:regularity}, we see that
\begin{equation}\label{7DG}
\|\alpha_{n}^{\prime}\|_{\infty} \leq C(1+|v_n|)\leq C,\  \text {for any } n.
\end{equation}
Hence from
Ascoli-Arzel{\`a} Theorem, we deduce that, after extracting a  subsequence, $\alpha_n$ uniformly converge to some $\alpha\in C([t,T];\re^{N})$.
In particular, 
calling $(x_n(\cdot), v_n(\cdot))$ the trajectory associated to $\alpha_n$ starting from $(x, v_n)$:
\begin{equation*}
x_{n}(s)=x+v_{n}(s-t)+\int_t^s \int_t^{\theta}\,\alpha_{n}(\tau)d\tau d\theta,\quad\hbox{and}\quad
v_{n}(s)=v_{n}+\int_t^s\,\alpha_{n}(\tau)d\tau.
\end{equation*}
we get:
\begin{eqnarray*}
&&x_{n}(s)\to x(s)=x+v(s-t)+ \int_t^s\int_t^{\theta}\,\alpha(\tau)d\tau d\theta,\ \text{uniformly in } [t, T],\\
&&v_{n}(s)\to v(s)= v+\int_t^s\,\alpha(\tau)d\tau\ \text{uniformly in }[t, T].
\end{eqnarray*}
Moreover, by classical arguments of stability, $\alpha$ is optimal, i.e. $\alpha\in\mathcal U(x,v, t)$.
The uniform convergence of the $\alpha_n$ yields  in particular that
$\alpha_n(t)\to \alpha(t)$ where $\alpha(t)$ is uniquely determined by assumption. By~(\ref{1DGv}) and \eqref{2DG}, we get that 
$\chi_n\to \chi=\alpha(t)$. This implies that $D_v^*u(x, v, t)$ is a singleton, then $D_vu(x,v, t)$ exists. Going back to the first part of the proof, we see that $D_vu(x,v, t)=-\alpha(t)$.

\medskip

2. If $\mathcal U(x,v, t)=\{\alpha(\cdot)\}$, then for any $s\in[t,T]$, $\alpha(s)$ is uniquely determined.
Indeed, if there exists $\beta\in \mathcal U(x(s),v(s), s)$, then the concatenation $\gamma$ of $\alpha$ and $\beta$ (see Lemma \ref{DPP}) is also optimal, i.e. $\gamma\in\mathcal U(x,v, t)=\{\alpha(\cdot)\}$.\\
Then from point 1 with $t=s$ at $(x(s),v(s))$, we deduce that $D_vu(x(s),v(s), s)$ exists.

\medskip

3. From point 1, we know that for any $\alpha(\cdot)\in {\mathcal{U}}(x,v, t)$, $\alpha(t)$ is unique and  coincides with~$-D_vu(x,v,t)$. Hence, relation~\eqref{tag:alpha*} ensures $p_v(t)=-D_vu(x,v,t)$.
On the other hand, note that, since $D_xu(x, v, t)$ exists, we get from (\ref{eq:1}) that
\begin{eqnarray*}
D_xu(x, v, t)&=&\int_t^T
D_x\ell(x(s),v(s),s)ds + D_xg(x(T), v(T))\\
&=&\int_t^Tp_x'(s)ds+ D_xg(x(T), v(T))=-p_x(t);
\end{eqnarray*}
thus, $p_x(t)$ and $p_v(t)$ are both uniquely determined.
Hence \eqref{tag:1}-\eqref{tag:4} is a system of differential equations with initial conditions $x(t), v(t)$, $p_x(t)$ and $p_v(t)$ 
which admits a unique solution $(x(\cdot), v(\cdot), p_x(\cdot),p_v(\cdot))$ by Cauchy-Lipschitz theorem, 
 and $(x(\cdot), v(\cdot))$ is the unique optimal  trajectory starting from $(x,v)$, associated to the unique optimal control law $\alpha(\cdot)=p_v(\cdot)$.
\end{proof}
\begin{lemma}[optimal synthesis]\label{B}
Consider $\xi\in \R^N$ and $\eta\in \R^N$.
  \begin{enumerate}
  \item Let $x\in C^1([t,T]; \R^N)$, $v\in {\rm{AC}}([t,T];\R^N)$ be such that
\begin{equation}
  \label{eq:2}
x(t)=\xi, \quad \hbox{and}\quad  v(t)=\eta,
\end{equation}
and for almost every $s\in (t,T)$,
\begin{equation}
  \label{eq:3}
u(x(s),\cdot,s) \hbox{ is differentiable at } v(s),
\end{equation}
and
\begin{equation}\label{OS}
  \begin{array}[c]{rcl}
x'(s)&=&v(s),\\
v'(s)&=& -D_vu(x(s),v(s), s),
  \end{array}
\end{equation}
where $u$ is the solution of \eqref{HJ}. Under these assumptions, the control law  $\alpha(s)=v'(s)=-D_vu(x(s),v(s), s)$ is optimal for $u(\xi,\eta, t)$.
\item If $u(\cdot, \cdot, t)$ is differentiable at $(\xi,\eta)$, 
 then problem (\ref{eq:2}), (\ref{OS}) has a unique solution corresponding to the optimal trajectory.
  \end{enumerate}
\end{lemma}

\begin{proof}
We adapt the arguments of \cite[Lemma 4.11]{C}. 
Fix $(t,\xi,\eta)\in(0,T)\times \re^{2N}$. Let $x\in C^1([t,T]; \R^N)$, $v\in {\rm{AC}}([0,T];\R^N)$ be as in the statement.
Note that, from \eqref{OS}, since $|D_vu|$ grows at most linearly in $v$ (see Lemma~\ref{L1}-(point $1$)), Gronwall's Lemma ensures that $v(\cdot)$ is bounded in $(t,T)$; consequently, again by \eqref{OS} and Lemma~\ref{L1}-(point $1$), $v(\cdot)$ is Lipschitz continuous.
Therefore, from Lemma \ref{L1}, the function $s\mapsto u(x(s),v(s),s)$ is Lipschitz continuous as well. Hence, for almost every $s\in [t,T]$, 
\begin{itemize}
\item (\ref{eq:3}) and (\ref{OS}) hold,
\item  the function $u(x(\cdot), v(\cdot),\cdot)$ admits a derivative at $s$.
\end{itemize}
Fix such an $s$. Lebourg's Theorem for Lipschitz functions (see \cite[Thm 2.3.7]{Cla90} and \cite[Thm 2.5.1]{Cla90}) ensures that, for any sufficiently small number $h$, there exists $(y_h, w_h, s_h)$ in the open line segment $( (x(s), v(s),s), (x(s+h), v(s+h), s+h))$ and  $(\chi^h_x, \chi^h_v,\chi^h_t) \in {\rm{conv}} \left(D_{x,v,t}^*u(y_h, w_h, s_h)\right)$ such that
\begin{equation}\label{31}
u(x(s+h),v(s+h), s+h)-u(x(s),x(s), s)= \chi^h_x\cdot  (x(s+h)-x(s))+ \chi^h_v\cdot (v(s+h)-v(s)) +\chi^h_t h.
\end{equation}
Here, $\rm{conv}(A)$ stands for the convex hull of a set $A$ while $D_{x,v,t}^*u(y_h, w_h, s_h)$ stands for the reachable gradient at $(y_h, w_h, s_h)$ with respect to the variables $x$, $v$ and $t$ (see~\cite[eq.~ (4.4)]{BCD}. \\
By  Carath{\'e}odory's theorem (see \cite[Thm A.1.6]{CS}), there exist $(\lambda^{h,i}, \chi^{h,i}_x, \chi^{h,i}_v, \chi^{h,i}_t)_{i=1,\dots,2N+2}$ such that   $\lambda^{h,i}\geq0$, $\sum_{i=1}^{2N+2}\lambda^{h,i}=1$, 
 $(\chi^{h,i}_x, \chi^{h,i}_v, \chi^{h,i}_t)\in D_{x,v, t}^*u(y_h,w_h, s_h)$ and $(\chi^h_x, \chi^h_v, \chi^h_t) = \sum_{i=1}^{2N+2}\lambda^{h,i}(\chi^{h,i}_x, \chi^{h,i}_v,\chi^{h,i}_t)$.
We claim that, for any $i=1,\dots, 2N+2$, there holds
\[\lim_{h\to 0}\chi^{h,i}_v=D_vu(x(s),v(s),s).
\]
Indeed, let $\chi^{i}_v$ be any cluster point of $\{\chi^{h,i}_v\}_h$. After a diagonal extraction, there exist $(x_n,v_n,t_n)$ such that $u$ is differentiable at $(x_n,v_n,t_n)$, $(x_n,v_n,t_n)\to (x(s),v(s),s)$ and $D_vu(x_n,v_n,t_n) \to \chi^{i}_v$ as $n\to\infty$.
By \cite[Lemma 4.6]{C} (applied to $z_n(\cdot):=u(x_n,\cdot,t_n)$), we have
\begin{equation*}
\chi^{i}_v=\lim_{n}D_vu(x_n,v_n,t_n)\in D^+ z(v(s))
\end{equation*}
where $z(\cdot):=u(x(s),\cdot,s)$. On the other hand, assumption~\eqref{eq:3} ensures that~$z$ is differentiable at $v(s)$; hence, by \cite[Proposition 3.1.5-(c)]{CS}, we get $\chi^{i}_v=D_v u(x(s),v(s),s)$ and our claim is proved. In particular, we deduce that $\chi^{h}_v$ converge to $D_vu (x(s),v(s), s)$ as $h\to 0$.\\
On the other hand, since $u$ is a viscosity solution to equation~\eqref{HJ} and $(\chi^{h,i}_x, \chi^{h,i}_v, \chi^{h,i}_t)\in D_{x,v, t}^*u(y_h,w_h, s_h)$, we obtain that for all $i\in 1,\dots, 2N+2$, 
\[
- \chi^{h,i}_t +\frac12\left|\chi^{h,i}_{v}\right|^2-\frac12\left|w_h\right|^2-w_h\cdot\chi^{h,i}_{x}=\ell(y_h, w_h, s_h).
\]
Therefore, 
$\chi^{h}_t  +   w_h\cdot \chi^{h}_x =  \frac12 \sum_{i=1}^{2N+2}\lambda^{h,i}\left|\chi^{h,i}_{v}\right|^2
-\frac12\left|w_h\right|^2- \ell(y_h,w_h,s_h)$ converges to\\
$\frac12 |D_v u(x(s),v(s), s)|^2  
- \frac12|v(s)|^2 -\ell(x(s),v(s), s)$ as $h\to 0$.
\\
Then dividing (\ref{31}) by $h$ and letting $h$ tend to $0$, we get that 
\begin{displaymath}
  \begin{split}
 & \frac {d}{ds} \left(u(x(s),v(s),s)\right)\\ =  
   & D_vu (x(s),v(s), s)\cdot v'(s)
 + \frac 1 2 \left|D_vu (x(s),v(s), s)\right| ^2 - \frac 1 2 |v(s)|^2 -\ell(x(s),v(s),s).    
  \end{split}
\end{displaymath}
Recalling (\ref{OS}), we get 
\begin{displaymath}
  \frac {d}{ds} \left(u(x(s),v(s),s)\right)= - \frac 1 2 \left|D_vu (x(s),v(s), s)\right| ^2 - \frac 1 2 |v(s)|^2 -\ell(x(s),v(s),s).
\end{displaymath}
or in equivalent manner,
\begin{displaymath}
  \frac {d}{ds} \left(u(x(s),v(s),s)\right)=  - \frac 1 2 \left|v'(s)\right| ^2 - \frac 1 2 |v(s)|^2 -\ell(x(s),v(s),s),
\end{displaymath}
which holds for almost every $s$.
Integrating this equality on $[t,T]$ and taking into account the terminal condition in~\eqref{HJ}, we obtain
\begin{displaymath}
u(x,v, t)=\int_t^T\frac12|v'(s)|^2+ \frac12|v(s)|^2+ \ell(x(s),v(s), s) ds +g(x(T),v(T)).  
\end{displaymath}
Therefore, the control law $\alpha(s)= v'(s)= -D_vu (x(s),v(s),s)$ is optimal. This achieves the proof of the first statement.

\medskip

The second statement is a direct consequence of  Lemma \ref{4.9}.
\end{proof}

\section{The continuity equation}\label{sect:c_eq}
In this section, our aim is to study  equation \eqref{eq:MFGAs}-(ii), and more precisely the well-posedness of 
\begin{equation}\label{continuity}
\left\{
\begin{array}{ll} \partial_t m+v\cdot D_xm-
\diver_v (m\, D_v u)= 0,&\qquad \textrm{in }\re^{2N}\times (0,T),\\
m(x,v, 0)=m_0(x,v), &\qquad \textrm{on }\re^{2N},
\end{array}\right.
\end{equation}
where $u$ is the value function associated to the cost $J_t$ in~\eqref{def:OC}; for the sake of clarity, let us recall from Proposition~\ref{exuniq} that $u$ is the unique viscosity solution fulfilling~\eqref{eq:stimau} to the problem
\begin{equation*}
\left\{\begin{array}{ll}
-\partial_t u-v\cdot D_xu+\frac1 2 \vert D_vu\vert^2-\frac1 2 \vert v \vert^2-l(x,v)  =F[\overline m(t)](x,v),&\quad \textrm{in }\R^{2N}\times (0,T),\\
u(x,v, T)=G[\overline m(T)](x, v),&\quad \textrm{on }\re^{2N},
\end{array}\right.
\end{equation*}
and~$\overline m$ is fixed and belongs to   $ C([0,T];\mathcal P_1(\re^{2N}))$.
It is worth to observe that the differential equation in~\eqref{continuity} can also be written
\[
\partial_t m- \diver_{x,v} (m\,  b)=0,
\]
with $ b:=(-v, D_v u)$.
In the present framework, the properties of $u$ (semi-concavity and local Lipschitz continuity) are not enough to ensure that the flow $\Phi(x,t,s)$  given by Lemma~\ref{B} has a Lipschitz continuous inverse,  by contrast with  \cite[Lemma 4.13]{C}.
Moreover, the drift $ b$ is only locally bounded; this lack of regularity makes it impossible to apply the standard results  for drifts which are Lipschitz continuous (uniqueness, existence and representation formula of $m$ as the push-forward of $m_0$ through the characteristic flow; e.g., see \cite[Proposition 8.1.8]{AGS}). We shall overcome this difficulty by applying the superposition principle \cite[Theorem 8.2.1]{AGS}. The latter yields a representation formula of $m$ as the push-forward of some measure on~$C([0,T];\re^{2N})$ through the evaluation map~$e_t$ defined by  $e_t(\gamma)=\gamma(t)$ for all continuous function $\gamma$ with value in $\R^{2N}$.
 In the following theorem, we  state existence, uniqueness, and some regularity results for~\eqref{continuity}:

\begin{theorem}\label{prp:m}
Under assumptions {\rm (H)}, for any $\overline m\in  C([0,T];  \mathcal P_1(\re^{2N}))$,
 there is a unique $m\in C^{\frac 1 2} ([0,T];  \mathcal P_1(\re^{2N})) \cap L^\infty((0,T);\mathcal P_2(\re^{2N}))$
which solves problem \eqref{continuity}   in the sense of Definition~\ref{defsolmfg}.

Moreover $m(t,\cdot)$ satisfies: for any for $\phi\in C^0_b(\R^{2N})$, for any $t\in [0,T]$,
\begin{equation}
\label{ambrosio}
\int_{\re^{2N}} \phi(x,v)\, m(x,v,t)dxdv=\int_{\re^{2N}}\phi\left(\overline {\gamma}_{x,v}(t)\right)\,m_0(x,v)\, dx dv, 
\end{equation}
where, for a.e. $(x,v)\in\re^{2N}$,  $\overline{\gamma}_{x,v}$ is the solution to \eqref{dyn}.
\end{theorem}

The proof of  Theorem~\ref{prp:m} is given in the next two subsections which are devoted respectively to  existence (see Proposition~\ref{VV}) and to uniqueness  and the representation formula (see Proposition~\ref{!FP}).
\subsection{Existence of the solution}\label{subsect:ex}

We wish to  establish the existence of a solution to the continuity equation
via a vanishing viscosity method applied to the {\it whole} MFG system in which the viscous terms involve
Laplace operators with respect to {\it both} $x$ and $v$. This is reminiscent of   \cite[Appendix]{C13} (see also \cite[Section 4.4]{C}).
In this way, $D_vu$ is replaced by $D_vu^\sigma$, which is regular by standard regularity theory for parabolic equations; 
this implies the regularity of the solution of the Fokker-Planck equation (see \cite{CH}).
 Note also that $D_vu$ may be unbounded; we shall overcome this issue by taking advantage of estimates similar to  those in Lemma \ref{L1}. 
Indeed, these estimates will allow us to apply classical results for the existence and uniqueness of the solution.

\begin{proposition}\label{VV}
Under assumptions $\rm{(H)}$, for any~$\overline m\in C([0,T];\mathcal P_1(\re^{2N}))$, problem \eqref{continuity} has a solution $m$ in the sense of Definition \ref{defsolmfg}. Moreover $m\in C^{\frac 1 2} ([0,T];  \mathcal P_1(\re^{2N})) \cap L^\infty(0,T;\mathcal P_2(\re^{2N}))$.
\end{proposition}

We consider the solution $(u^\sigma, m^\sigma)$ to the following problem
\begin{equation}
\label{eq:MFGv}
\left\{
\begin{array}{lll}
(i)\ -\partial_t u-\sigma \Delta_{x,v} u-v\cdot D_xu+\frac1 2 \vert D_vu\vert^2-\frac1 2 \vert v \vert^2-l(x,v)  =F[\overline m](x, v),\ & \textrm{in }\re^{2N}\times (0,T),\\
(ii)\ \partial_t m-\sigma \Delta_{x,v} m-\diver _v (m D_v u)-v\cdot D_xm=0, & \textrm{in }\re^{2N}\times (0,T),\\
(iii)\ m(x,v, 0)=m_0(x,v),\quad  u(x,v,T)=G[\overline m(T)](x, v),& \textrm{on }\re^{2N}.
\end{array}\right.
\end{equation}

Recall that equation \eqref{eq:MFGv}-(ii) has a standard probabilistic interpretation (see relation~\eqref{mstoch} below).
Our aim is to find a solution to problem~\eqref{continuity} by letting $\sigma$ tend to  $0^+$. To this end, some estimates are needed.
\\
Note that equation~\eqref{eq:MFGv}-(ii) can be written in the compact form
\begin{equation}\label{eq:4}
\partial_t m-\sigma \Delta_{x,v} m-\diver _{x,v} (m { b}^\sigma)=0,\qquad \textrm{with }{ b}^\sigma:=(-v, D_v u).
\end{equation}
 
We start by establishing the well-posedness of  system~\eqref{eq:MFGv} and that the functions $u^\sigma$ are Lipschitz continuous and semi-concave uniformly in~$\sigma$.

\begin{lemma}\label{visco:lemma5.2}
Under the same assumptions as in Proposition~\ref{VV}, there exists a unique classical solution~$u^\sigma$ to problem~\eqref{eq:MFGv}-(i), -(iii) with quadratic growth in $(x,v)$. Moreover, there exists a constant $C>0$
which depends only on the constants in assumptions $\rm{(H)}$, in particular  it is independent of $\sigma\le 1$,
 such that
\begin{eqnarray*}
&(a)&|u^\sigma(x,v,t)|\leq C(1+|v|^2),\\
&(b)&\|D_x u^\sigma\|_\infty\leq C, \quad |D_v u^\sigma(x,v,t)|\leq C(1+|v|), \quad |\partial_t u^\sigma(x,v,t)|\leq C(1+|v|^2),\\
&(c)& D^2_{x,v} u^\sigma\leq C,
\end{eqnarray*}
where $D_{x,v}^2 u$ is the Hessian of~$u$ with respect to both~$x$ and~$v$.
\end{lemma}
\begin{proof}
Following the same arguments of Proposition~\ref{exuniq} (based on the comparison principle by Da Lio and Ley~\cite[Theorem 2.1]{DLL}),
 one can easily prove the existence  of a viscosity solution to equation ~\eqref{eq:MFGv}-(i) with terminal condition 
as in ~\eqref{eq:MFGv}-(iii) and   satisfying inequality~$(a)$. Furthermore, still by the results in~\cite{DLL}, this solution is unique among the functions with this growth at infinity. Hence, estimate~$(a)$ is proved.

Let us now prove that this viscosity solution~$u^\sigma$ is a classical solution.   To this end, let us  assume for a moment that $u^\sigma$ satisfies estimates~$(b)$ and $(c)$.   We see that $u$ is a viscosity subsolution of 
\begin{displaymath}
- \partial_t u-\sigma \Delta_{x,v} u-v\cdot D_xu\le C(1+|v|^2).
\end{displaymath} 
Moreover, from estimate~$(c)$, we see that at any point $(x,v,t)$, either $u^\sigma$ is twice differentiable with respect to $x$ and $v$, or 
there does exist a smooth function that touches   $u^\sigma$  from below.  This and  estimate~$(b)$ imply  that $u^\sigma$  is a viscosity supersolution of 
\begin{displaymath}
- \partial_t u-\sigma \Delta_{x,v} u-v\cdot D_xu\ge -C(1+|v|^2).
\end{displaymath}
for some positive constant $C$. From \cite{MR1341739}, $u$ is also a distributional subsolution (respectively supersolution) of the same linear inequalities.
Therefore,  both $-\partial_t u^\sigma-\sigma \Delta_{x,v} u^\sigma-v\cdot D_xu^\sigma$ and 
$-\frac1 2 \vert D_vu^\sigma\vert^2+\frac1 2 \vert v \vert^2 +l(x,v)+F[\overline m](x, v)$ are in $L_{\rm loc}^\infty$.
On the other hand, from~$(b)$ and $(c)$,  Alexandrov's theorem implies that $u^\sigma$ is twice differentiable with respect to $x$ and $v$ almost everywhere, so the equation 
$$-\partial_t u^\sigma-\sigma \Delta_{x,v} u^\sigma-v\cdot D_xu^\sigma
=-\frac1 2 \vert D_vu^\sigma\vert^2+\frac1 2 \vert v \vert^2 +\ell(x,v,t),$$
(where $\ell$ and $g$ are defined in (\ref{ell})),
holds almost everywhere, and in the sense of distributions since both the left and right hand sides are in $L^\infty_{\rm loc}$.
\\
Hence classical results on the regularity of weak solutions (including bootstrap) can be applied and yield that $u$ is a classical solution. 

Let us now prove the estimates $(b)$ and $(c)$, by using similar arguments  to those contained in the proofs 
of Lemma  \ref{L1}. They  use a representation formula of $u$ arising from  a stochastic optimal control problem (see, for example, \cite{DLL,BCQ,C}).
\\
Let $(\Omega, \mathcal F, (\mathcal F_t), \mathbb{P})$ be a  complete filtered probability space, the filtration $(\mathcal F_t)$ supporting a standard $2N$-dimensional Brownian motion $B_s=(B_{x,s}, B_{v,s})$.
Let ${\mathcal A}_t$ be the set of $\R^{N}$-valued $(\mathcal F_t)$-progressively measurable processes and  let  $\mathbb{E}$ be the expectation with respect to the probability measure $\mathbb{P}$. The unique solution of \eqref{eq:MFGv}-(i) which satisfies point (a) can be written as:
\begin{displaymath}
u^\sigma(x,v, t) =\inf_{\alpha\in \mathcal A_t} \mathbb{E}\left(
\begin{array}[c]{l}
\ds \int_t^T\left[\frac12 |\alpha(s)|^2+\frac12 |V(s)|^2+\ell(X(s), V(s), s) \right] \,ds +g(X(T), V(T))
  \end{array}
     \right)
\end{displaymath}
where the controlled process $(X(\cdot), V(\cdot))$  satisfies 
\begin{displaymath}
  X(t)=x, \quad V(t)=v,
\end{displaymath}
almost surely and 
is governed by the stochastic differential equations
\begin{equation}
\label{1-stoc}
\left\{
\begin{array}{l}
dX=V(s) ds +\sqrt{2\sigma} dB_{x,s}, \\
dV= \alpha(s) ds +\sqrt{2\sigma} dB_{v,s}.
\end{array}
\right.
\end{equation}
Thus, almost surely,
\begin{equation}
\label{pro}
\left\{
\begin{array}{rcl}
\ds X(s)&=& \ds x+v(s-t)+\int_t^s \int_t^{\theta}\alpha(\tau) \,d\tau d\theta+\sqrt{2\sigma} \int_t^s\left( \int_t^{\theta}
dB_{v,\tau}\right)\, d\theta+\sqrt{2\sigma} \int_t^s dB_{x,\tau},\\
V(s)&=& \ds v+\int_t^s \alpha(\tau)\,d\tau+\sqrt{2\sigma} \int_t^s dB_{v,\tau}.
\end{array}
\right.
\end{equation}
To prove $(b)$, we can exactly use the same arguments as for Lemma \ref{L1}, replacing the paths 
$(x(s), v(s))$ and $(y(s), w(s))$ by the processes $(X(s), V(s))$ and $(Y(s), W(s))$,
 and noting that, from \eqref{pro}, we get similar equalities as in \eqref{diff}.\\
Note that, for any $\sigma$,  we get from $(a)$ that any $\epsilon$-optimal control $\alpha^{\sigma}$ for $u^\sigma(x,v, t)$ satisfies
\begin{displaymath}
\mathbb{E}\left(\int_t^T|\alpha^{\sigma}(s)|^2ds\right)\leq C(1+|v|^2),  
\end{displaymath}
hence, we get the same estimates as  (\ref{eq:9})
and \eqref{L2}, namely   estimate $(b)$. An analytic proof of $(b)$ is also possible,  see \cite[Chapter XI]{Lie}. 
\\

In order to prove $(c)$, we can follow the same procedure as in the proof of  Lemma \ref{semi-concav},
noting that:\\
i) equalities \eqref{eqsemi} and \eqref{reltutte} are still true for the stochastic processes,\\
ii) if we fix $s\in [t,T]$, using a  Taylor expansion of $g$ as in \eqref{taypro}, we get 
\begin{multline*}
g(X(s),V(s))=
g(X_{\lambda}(s), V_{\lambda}(s))+ Dg(X_{\lambda}(s), V_{\lambda}(s))(X(s)-X_{\lambda}(s), V(s)-V_{\lambda}(s))\\
+\frac{1}{2}(X(s)-X_{\lambda}(s), V(s)-V_{\lambda}(s))D^2g(\xi,\eta)(X(s)-X_{\lambda}(s), V(s)-V_{\lambda}(s))^T,
\end{multline*}
where
$\xi=X(s)+\theta_1 (X(s)-X_{\lambda}(s))$, $\eta=V(s)+\theta_2 (V(s)-V_{\lambda}(s))$
for suitable $\theta_1$ and $\theta_2$ in $[0,1]$. For a similar proof, see  \cite{BCQ}.
\end{proof}

\begin{lemma}\label{visco:buonapos}
Under the same assumptions as in Proposition~\ref{VV}, there exists a unique classical solution $m^\sigma$ 
to problem~\eqref{eq:MFGv}-(ii), -(iii) with a sub-exponential growth in $(x,v)$. Moreover, $m^\sigma>0$.
\end{lemma}
\begin{proof}
By Lemma~\ref{visco:lemma5.2}, the problem for $m^\sigma$ can be written 
\[
\quad \partial_t m-\sigma \Delta_{x,v} m -b^{\sigma}\cdot D_{xv} m-\left(\Delta_{v}u^{\sigma}\right)m=0,
\quad m(0)=m_0,
\]
where $b^\sigma$ has been introduced in (\ref{eq:4}) and from the estimates contained in  Lemma \ref{visco:lemma5.2},
 $|b^{\sigma}|\leq C(1+|v|)$ and $\Delta_{v}u^{\sigma}\leq C$.
Using this  and the results contained in \cite{IK}, we get the existence and uniqueness of a classical solution $m^\sigma$ of \eqref{eq:MFGv}-(ii) with initial condition  as in~\eqref{eq:MFGv}-(iii).
From the assumptions on $m_0$ and Harnack inequality  (see for example \cite[Theorem 2.1, p.13]{LSU}) we get that $m^\sigma(\cdot,t)>0$ for $t>0$.
\end{proof}

Let us now prove some properties of the functions~$m^\sigma$ which will play a crucial role in the proofs of Proposition~\ref{VV} and  of Theorem~\ref{thm:main}.
\begin{lemma}\label{visco:lemma4}
Under the same assumptions of Proposition~\ref{VV}, there exists a constant $K>0$
which depends only on the constants in assumptions $\rm{(H)}$ and on $m_0$, in particular  it is independent of $\sigma\le 1$,  such that:
\begin{equation*}
\begin{array}{ll}
1.\quad &\|m^\sigma\|_\infty\leq K, \\
2.\quad &{\bf d}_1(m^\sigma(t_1),m^\sigma(t_2))\leq K(t_2-t_1)^{1/2}, \qquad \forall t_1\le t_2\in[0,T],\\
3.\quad & \displaystyle\int_{\re^{2N}}(|x|^2+|v|^2)\,dm^\sigma(t)(x,v)\leq K \left(\displaystyle\int_{\re^{2N}}(|x|^2+|v|^2)\,dm_0(x,v)+1 \right),\qquad \forall t\in[0,T].
\end{array}
\end{equation*}
\end{lemma}
\begin{proof}
Point 1. In order to prove this $L^\infty$ estimate, we  argue as in \cite[Theorem 5.1]{C13}. We note that
\begin{displaymath}
\diver _v (m^\sigma D_v u^\sigma)=D_vm^\sigma\cdot D_vu^\sigma +m^\sigma(\Delta_{v}u^\sigma)\leq D_vm^\sigma\cdot D_vu^\sigma +Cm^\sigma,  
\end{displaymath}
because of the semi-concavity of~$u$ established in Lemma~\ref{visco:lemma5.2} and the positivity of $m^\sigma$. Therefore, from assumption ${\rm (H2)}$, the function~$m^\sigma$ satisfies
\begin{displaymath}
\partial_t m^\sigma-\sigma \Delta_{x,v} m^\sigma-v \cdot D_xm^\sigma-D_vu^\sigma\cdot D_vm^\sigma -Cm^\sigma\leq 0,\qquad m^\sigma(x,v, 0)\leq C.  
\end{displaymath}

Then,  using $w=Ce^{ C t}$ as a supersolution (recall that $C$ is independent of $\sigma$), we obtain that 
$\|m^\sigma\|_\infty\leq Ce^{ C T}$, using the comparison principle proved in~\cite[Theorem 2.1]{DLL}.

To prove Points 2 and 3  as in the proof of \cite[Lemma 3.4 and 3.5]{C}, it is convenient to introduce the stochastic differential equation
\begin{equation}\label{11C}
dY_t=  b^\sigma(Y_t,t) dt +\sqrt{2\sigma} dB_t,\qquad Y_0=Z_0,
\end{equation}
where  $Y_t=(X_{t}, V_{t})$,  $ b^\sigma (x,v,t)= (-v, D_vu^\sigma(x,v,t))$,
$B_t$ is a standard $2N$-dimensional Brownian motion, and ${\mathcal L}(Z_0)=m_0$. By standard arguments, setting
\begin{equation}
\label{mstoch}
m^\sigma(t):={\mathcal L}(Y_t),
\end{equation}
we know that $m^\sigma(t)$ is absolutely continuous with respect to Lebesgue measure, and that if $m^\sigma(\cdot, \cdot,t)$ is the density of $m^\sigma(t) $, then $m^\sigma$   is the  weak solution to~\eqref{eq:MFGv}-(ii) with $m^\sigma|_{t=0}=m_0$ (from Ito's Theorem, since $ b^\sigma$ has at most linear growth with respect to $(x,v)$, Proposition 3.6 Chapter 5 \cite{KS}, p.303, and the book \cite{Kr}). Here  again, we have used the estimate on $|D_vu^\sigma|$ given in Lemma \ref{visco:lemma5.2}.
 
\begin{description}
  \item{Point 3:}
    Noting that
    $$\int_{\re^{2N}}(|x|^2+|v|^2) dm^\sigma(t)(x,v)= \EEE(|Y_t|^2),$$ the desired estimate
    can be obtained by applying Estimate 3.17 of Problem 3.15, p. 306,  (the solutions are at p. 389) of \cite{KS} with $m=1$.
    \item{Point 2:}
For $t_2\geq t_1$, it is well known that
$${\bf d}_1(m^\sigma(t_1),m^\sigma(t_2))\leq  \EEE(|Y_{t_1}-Y_{t_2}|).$$
Recall also that for a suitable constant $C$,
$$\vert b^\sigma(Y_\tau, \tau)\vert \leq C(\vert V_\tau\vert+1).$$
  The latter two observations imply that
\begin{displaymath}
  \begin{array}[c]{rcl}
\mathbb{E}(|Y_{t_1}-Y_{t_2}|)&\leq& \ds  \mathbb{E}\left(\int_{t_1}^{t_2} \vert b^\sigma(Y_\tau, \tau)\vert d\tau+ \sqrt {2\sigma}|B_{t_2}-B_{t_1}|\right)\\
&\leq&  \ds  \mathbb{E}\left(C \int_{t_1}^{t_2} (\vert V_\tau \vert +1)\vert d\tau+ \sqrt {2\sigma} |B_{t_2}-B_{t_1}|\right)\\
&\leq& \ds C\left(\mathbb{E}\left(\int_{t_1}^{t_2} (\vert V_\tau\vert ^2 +1)\vert  d\tau\right)\right) ^{\frac 1 2 } \sqrt{t_2-t_1} + \sqrt {2\sigma}\sqrt{t_2-t_1} \\
&\leq & \ds C\left(\mathbb{E}\left(\max_{[{t_1},{t_2}]}\vert Y_\tau\vert^2\right) +1\right) ^{\frac 1 2 }     (t_2-t_1)+ \sqrt {2\sigma}\sqrt{t_2-t_1} .
\end{array}
\end{displaymath}
where we have used estimate \cite[(3.17) p. 306]{KS}.
\end{description}
\end{proof}

\begin{proof}[Proof of Proposition \ref{VV}]
The arguments are similar to those in  the proof of \cite[Theorem 5.1]{C13} (see also \cite[Theorem 4.20]{C}). 
Lemma~\ref{visco:lemma5.2} implies that possibly after  the extraction of a subsequence, 
$u^\sigma$  locally uniformly converges to some function~$u$, 
which is Lipschitz continuous with respect to $x$, locally Lipschitz continuous with respect to $v$, and 
$Du^\sigma\to Du$ a.e. (because of the semi-concavity estimate of Lemma~\ref{visco:lemma5.2} and \cite[Theorem 3.3.3]{CS}).
By standard stability result for viscosity solutions, the function~$u$ is a viscosity solution of  \eqref{HJ}.

On the other hand, the function $m^\sigma$ satisfies the estimates stated in Lemma~\ref{visco:lemma4}:
\begin{enumerate}
\item from point 3,  $m^\sigma(t)$ is bounded in $\mathcal P_2(\R^{2N})$ uniformly in $\sigma\in[0,1]$ and $t\in [0,T]$
\item from points 2 and 3 , $m^\sigma$ is  bounded in  $C^{1/2}([0,T];\mathcal P_1(\R^{2N}))$ uniformly with respect to $\sigma\in[0,1]$.
\end{enumerate}
Recalling that  the subsets of   $\mathcal P_1(\R^{2N})$ whose  elements have uniformly bounded second moment
  are relatively compact in $\mathcal P_1(\R^{2N})$, see for example \cite[Lemma 5.7]{C}, we can apply Ascoli-Arzel{\`a} theorem:
we may extract a sequence (still indexed by~$\sigma$ for simplicity) such that $\sigma\to 0^+$ and   $m^\sigma$ converges to some $m\in C^{1/2}([0,T];\mathcal P_1(\R^{2N}))$ in the $C([0,T];\mathcal P_1(\R^{2N}))$ topology. Moreover, from point 1 in Lemma~\ref{visco:lemma4} and  Banach-Alaoglu theorem, $m$ belongs to 
$L^\infty_{\rm loc}((0,T)\times\re^{2N})$ and the sequence $m^\sigma$ converges to $m$ in $L^\infty_{\rm loc}((0,T)\times\re^{2N})$-weak-$*$.
\\
Therefore, by passing to the limit, we immediately obtain that $m|_{t=0}=m_0$, $\|m\|_\infty \le K$ and that
${\bf d}_1(m(t_1),m(t_2))\leq K(t_2-t_1)^{1/2}$, $\forall t_1\le t_2\in[0,T]$.
\\
Let us prove that for all $t\in [0,T]$,
\begin{equation}
  \label{eq:7}
\ds \int_{\re^{2N}}(|x|^2+|v|^2)\,dm(t)(x,v)\leq K \left(\int_{\re^{2N}} (|x|^2+|v|^2)\,dm_0(x,v)+1 \right).
\end{equation}
For that, let us consider the increasing  sequence of functions defined on $\R_+$: $\phi_n(\rho)= 1\wedge ((n+1-\rho)\vee 0)$.
We know from point 3 in Lemma~\ref{visco:lemma4}, that for all $t\in [0,T]$,
\begin{equation}
  \label{eq:5}
\ds \int_{\re^{2N}}(|x|^2+|v|^2)   \phi_n(|x|^2+|v|^2)  m^\sigma(x,v,t)   dxdv 
\leq K \left(\int_{\re^{2N}} (|x|^2+|v|^2)\,dm_0(x,v)+1 \right).
\end{equation}
For a fixed $n$, we can pass to the limit in (\ref{eq:5}) thanks to the $L^\infty_{\rm loc}((0,T)\times\re^{2N})$-weak-$*$ convergence established above. We obtain:
    \begin{equation}
\label{eq:6}
\ds \int_{\re^{2N}}(|x|^2+|v|^2)   \phi_n(|x|^2+|v|^2)  m(x,v,t)   dxdv 
\leq K \left(\int_{\re^{2N}}|(|x|^2+|v|^2)\,dm_0(x,v)+1 \right).
\end{equation}
We then pass to the limit as $n\to +\infty$ thanks to Beppo Levi monotone convergence theorem, and obtain (\ref{eq:7}).
\\
Finally,  $m^\sigma$ is a solution to \eqref{eq:MFGv}-(ii),
\begin{displaymath}
\int_0^T\int_{\re^{2N}}m^\sigma\left(-\partial_t \psi -\sigma \Delta_{x,v} \psi+D_v\psi\cdot D_vu^\sigma - v \cdot D_x\psi \right)\,dxdv\, dt=0  
\end{displaymath}
 for any $\psi\in C^\infty_c((0,T)\times\re^{2N})$.
Applying the dominated convergence theorem, we infer $D_v\psi\cdot D_vu^\sigma\to D_v\psi\cdot D_vu$ in $L^1$ as $\sigma \to 0^+$ because $D_v u^\sigma$ are locally bounded (see Lemma~\ref{visco:lemma5.2}), $D_vu^\sigma\to D_vu$ a.e. and $\psi$ has a compact support.
 Letting $\sigma \to 0^+$, we conclude from the $L^\infty_{\rm loc}$-weak-$*$ convergence of~$m^\sigma$ and the convergence  $Du^\sigma\to Du$ a.e. that the function~$m$ solves \eqref{continuity} in the sense of Definition~\ref{defsolmfg}.
\end{proof}

\begin{remark}
  \label{sec:existence-solution}
Note that we have just proven that all the estimates  on $u^\sigma$ contained in Lemma \ref{visco:lemma5.2} hold for $u$. 
These estimates have also been obtained directly in the proof of Lemma~\ref{L1}. Similarly,  all the estimates  on $m^\sigma$ contained in Lemma \ref{visco:lemma4} hold for $m$. 
\end{remark}

\subsection{Uniqueness of the solution}\label{uniq}

We now deal with uniqueness for  ~\eqref{continuity}.
\begin{proposition}\label{!FP}
Under assumptions $\rm{(H)}$, the function $m$ found in Proposition \ref{VV} is the  unique 
solution to problem \eqref{continuity}   in the sense of Definition~\ref{defsolmfg} such that \\
$m\in C^{\frac 1 2} ([0,T];  \mathcal P_1(\re^{2N})) \cap L^\infty((0,T);\mathcal P_2(\re^{2N}))$.
\\
Moreover, $m$ satisfies:
\begin{equation}\label{ambrosio2}
\int_{\re^{2N}} \phi(x,v)   \,  m(x,v,t) dxdv 
=\int_{\re^{2N}}\phi(\overline{ \gamma}_{x,v}(t))\,m_0(x,v)\, dxdv, \qquad \forall \phi\in C^0_b(\R^{2N}), \, \forall t\in[0,T],
\end{equation}
where, for a.e. $(x,v)\in\re^{2N}$,  $\overline{\gamma}_{x,v}$ is the solution to \eqref{dyn}.
\end{proposition}
\begin{proof}[Proof of Proposition \ref{!FP}]
The proof is similar to that of \cite[Proposition A.1]{CH}, which relies on the superposition principle~\cite[Theorem 8.2.1]{AGS}.
Let $\Gamma_T$ denote the set of continuous curves in $\R^{2N}$, namely $\Gamma_T=C([0,T];\re^{2N})$.
 For any $t\in[0,T]$, we introduce the evaluation map: $e_t: \R^{2N}\times\Gamma_T\to \re^{2N}$, $e_t(x,v,\gamma):=\gamma(t)$.
Hereafter, when we write ``for a.e.'' without specifying the measure, we intend ``with respect to the Lebesgue measure''.

Let $m\in C^{1/2}([0,T];{\mathcal P}_1(\re^{2N}))\cap L^\infty((0,T);{\mathcal P}_2(\re^{2N}))$ 
be a solution of problem~\eqref{continuity} in the sense of Definition~\ref{defsolmfg}. Recall the notation $b(x,v,t)= (-v, D_vu(x,v,t))$.
The estimate  (8.1.20) in chapter 8 of \cite{AGS} is fulfilled: indeed,  
\begin{eqnarray*}
\int_0^T\int_{\re^{2N}}|b(x,v,t)|^2dm(t)(x,v)&\leq& C\int_0^T\int_{\re^{2N}}|v|^2dm(t)(x,v)\\
&&+C\int_0^T\int_{\re^{2N}}|D_vu(x,v,t)|^2dm(t)(x,v)\leq C,
\end{eqnarray*}
where the last inequality comes from the estimates on $D_vu$ and $m$ in Remark 4.2 (recall that $m(t)$ is a probability measure).
Therefore, the assumptions of the superposition principle are fulfilled (see \cite[Theorem 8.2.1]{AGS} and also \cite[pag. 182]{AGS}). 
 The latter and the disintegration theorem (see \cite[Theorem 5.3.1]{AGS}) entail that there exist a probability measure 
$\eta$  on $\R^{2N}\times \Gamma_T$   and for $m_0$-almost every  $(x,v)\in\re^{2N}$, a probability measure on $\eta_{x,v}$ on~$\Gamma_T$, such that
\begin{description}
  \item{i)} $e_t\#\eta =m_t $, i.e., for every bounded and continuous real valued function $\psi$ defined on $\R^{2N}$, for every $t\in [0,T]$,
    \begin{displaymath}
      \int_{\R^{2N}} \psi(x,v)dm_t(x,v)= \int_{\R^{2N}\times \Gamma_T} \psi(\zeta(t)) d\eta (x,v,\zeta).
    \end{displaymath}
In particular,  $ e_0\#\eta =m_0$.
 \item{ii)}
   \begin{displaymath}
\eta =\displaystyle\int_{\re^{2N}}\eta_{x,v}\, dm_0(x,v), 
   \end{displaymath}
i.e. for every  bounded Borel function $f: \R^{2N}\times \Gamma_T \to \R$, 
\begin{displaymath}
  \int_{\R^{2N}\times \Gamma_T} f(x,v, \zeta) d\eta (x,v,\zeta)= \int_{\R^{2N}} \left(\int_{\Gamma_T}   f(x,v,\zeta) d\eta_{x,v}(\zeta)
\right) dm_0(x,v).
\end{displaymath}

\item{iii)} For $m_0$-almost every $(x,v)\in \R^{2N}$, the support of $\eta_{x,v} $ is contained in the set 
  \begin{equation}
    \label{eq:8}
\left   \{  \zeta\in {\rm AC} \left([0,T];  \R^{2N} \right) :    \zeta(t) = (\xi(t), \eta(t)): \left|
     \begin{array}[c]{l}
       \xi(0)=x,\;\eta(0)=v,\\
       \xi'(t)=\eta(t),\\
       \eta'(t)= -D_vu(\xi(t),\eta(t),t). 
     \end{array}
\right.  \right\}.
 \end{equation}
  \end{description}
Recall that in the present case, $m_0$ is absolutely continuous (from assumption ${\rm (H4)}$); hence, since for all $t\in [0,T]$, $u(\cdot,\cdot, t)$ is Lipschitz continuous, the optimal synthesis in Lemma~\ref{B} ensures that for a.e. $(x,v)\in\re^{2N}$, (\ref{eq:2})-\eqref{OS} (with $t=0$ in the present context) has a unique  solution $\overline{\gamma}_{x,v}$, because it is the optimal trajectory for the cost $J_t$. 
 Therefore, for a.e. $(x,v)\in\re^{2N}$,  the set in (\ref{eq:8}) is a singleton, or in equivalent manner, 
$\eta_{x,v}$ coincides with $\delta_{\overline{\gamma}_{x,v}}$.
In conclusion, for any function $\psi\in C^0_b(\re^{2N})$, 
\begin{displaymath}
\begin{array}[c]{rcl}
\ds \int_{\re^{2N}} \psi (x,v)\, m(x,v,t) dxdv&=&\ds \int_{\R^{2N}\times\Gamma_T} \psi( e_t(\zeta)) d\eta(x,v,\zeta)
\\
&=&\ds \int_{\R^{2N}} \left(\int_{\Gamma_T}    \psi( e_t(\zeta)) d\eta_{x,v}(\zeta)
\right) dm_0(x,v)\\
&=& \ds 
\int_{\R^{2N}}    \psi( e_t(\overline{\gamma}_{x,v})) dm_0(x,v)\\
&=& \ds \int_{\R^{2N}}    \psi(\overline{\gamma}_{x,v} (t)) m_0(x,v) dxdv.
  \end{array}
\end{displaymath}
This shows that $m$ is uniquely defined as the image of $m_0$ by the flow of (\ref{dyn}).
\end{proof}

\begin{proof}[Proof of Theorem~\ref{prp:m}]{\empty}
Existence of $m$ comes from Proposition~\ref{VV},  uniqueness and the representation formula come from Proposition~\ref{!FP}.
\end{proof}

%
%
%
%
\section{Proof of the main results}\label{sect:MFG}
\begin{proof}[Proof of Theorem~\ref{thm:main}]{\empty}
For point 1, we  argue as in the proof of \cite[Theorem 4.1]{C}. Consider the set ${\mathcal C} :=\{m\in C([0,T]; {\mathcal P}_1(\R^{2N}))\mid m(0)=m_0\}$ endowed with the norm of~$C([0,T]; {\mathcal P}_1(\R^{2N}))$ and observe that it is a closed and convex subset of~$C([0,T]; {\mathcal P}_1(\R^{2N}))$. We also introduce a map $\cT$ as follows: to any $m\in {\mathcal C}$, we associate the solution~$u$ to problem~\eqref{HJ} with $\overline m=m$ and to this $u$ we associate the solution~$\mu=:\cT(m)$ to problem \eqref{continuity} which, by Proposition~\ref{VV} belongs to~${\mathcal C}$. Hence,~$\cT$ maps~${\mathcal C}$ into itself.
We claim that the map~$\cT$ has the following properties:
\begin{itemize}
\item[(a)] $\cT$ is a continuous map with respect to the norm of~$C([0,T]; {\mathcal P}_1(\R^{2N}))$
\item[(b)] $\cT$ is a compact map.
\end{itemize}
Assume for the moment that these properties are true. In this case, Schauder fixed point Theorem ensures the existence of a fixed point for~$\cT$, namely a solution to system~\eqref{eq:MFGAs}. Therefore it remains to prove properties $(a)$ and $(b)$.

Let us now prove~$(a)$.  Let~$(m_n)_n$ be a sequence in~ ${\mathcal C}$ such that $m_n\to m$ in the $C([0,T]; {\mathcal P}_1(\R^{2N}))$ topology. We want to prove that $\cT(m_n)\to \cT(m)$ in  $C([0,T]; {\mathcal P}_1(\R^{2N}))$.
We observe that hypothesis ${\rm (H3)}$ ensures that the functions~$(x,v,t)\mapsto F[m_n(t)](x,v)$ and~$(x,v)\mapsto G[m_n(T)](x,v)$ converge locally uniformly to the map~$(x,v,t)\mapsto F[m(t)](x,v)$ and respectively~$(x,v)\mapsto G[m(T)](x,v)$. Moreover, Lemma~\ref{L1} entails that the solutions~$u_n$ to problem~\eqref{HJ} with $\overline m=m_n$ are locally uniformly bounded and locally uniformly Lipschitz continuous. Therefore, by standard stability results for viscosity solutions, the sequence~$(u_n)_n$ converges locally uniformly to viscosity the solution~$u$ to problem~\eqref{HJ} with $\overline m=m$. Moreover, from Lemma~\ref{semi-concav}, the functions~$u_n$ are uniformly semi-concave; hence, by~\cite[Theorem 3.3.3]{CS}, $D u_n$ converge a.e. to $Du$.

By Proposition~\ref{VV} and Remark~\ref{sec:existence-solution}, the function~$\cT(m_n)$ verifies the bounds in Lemma~\ref{visco:lemma4} with a constant~$K$ independent of~$n$. Hence, the sequence~$(\cT(m_n))_n$ is uniformly bounded in~$C([0,T]; {\mathcal P}_1(\R^{2N}))$ (by Lemma~\ref{visco:lemma4}-(3) and Remark \ref{sec:existence-solution}, and because the subsets of   $\mathcal P_1(\R^{2N})$ whose  elements have uniformly bounded second moment are relatively compact in $\mathcal P_1(\R^{2N})$), and uniformly H{\"o}lder continuous in time with values in  ${\mathcal P}_1(\R^{2N})$  (by Lemma~\ref{visco:lemma4}-(2) and Remark \ref{sec:existence-solution}). Therefore, by Ascoli-Arzel{\`a} and Banach-Alaoglu theorems, there exists a subsequence~$(\cT(m_{n_k}))_k$ which converges to some $\mu\in C([0,T]; {\mathcal P}_1(\R^{2N}))$ in the $C([0,T]; {\mathcal P}_1(\R^{2N}))$-topology and in the $L^\infty_{\rm loc}((0,T)\times\re^{2N})$-weak-$*$ topology. As in Remark~\ref{sec:existence-solution}, $\mu$ verifies the bounds in~Lemma~\ref{visco:lemma4} and $\mu(0)=m_0$.

Observe that $\cT(m_{n_k})$ solves problem~\eqref{continuity} with $u$ replaced by~$u_{n_k}$, 
\begin{displaymath}
\int_0^T\int_{\re^{2N}}\cT(m_{n_k})\left(-\partial_t \psi +D_v\psi\cdot D_vu_{n_k}-v\cdot D_x\psi \right)\,dxdv\, dt=0,  
\end{displaymath}
for any $\psi\in C^\infty_c((0,T)\times\re^{2N})$. Passing to the limit as $k\to\infty$, we get that $\mu$ is a solution to~\eqref{continuity}. By the uniqueness result  established in Proposition~\ref{!FP}, we deduce that $\mu=\cT(m)$, and  that the whole sequence~$(\cT(m_n))_n$ converges to~$\cT(m)$.

Let us now prove $(b)$; since~${\mathcal C}$ is closed, it is enough to prove that $\cT({\mathcal C})$ is a precompact subset of $C([0,T]; {\mathcal P}_1(\R^{2N}))$. 
Let~$(\mu_n)_n$ be a sequence in~$\cT({\mathcal C})$ with $\mu_n=\cT(m_n)$ for some~$m_n\in{\mathcal C}$;
 we wish to prove that, possibly for a subsequence,  $\mu_n$ converges to some $\mu$ in the $C([0,T]; {\mathcal P}_1(\R^{2N}))$-topology as $n\to\infty$.

By Remark~\ref{sec:existence-solution}, the functions~$\cT(m_n)$ satisfy the estimates in Lemma~\ref{visco:lemma4} with the same constant~$K$. Since the subsets of   $\mathcal P_1(\R^{2N})$ whose  elements have uniformly bounded second moment are relatively compact in $\mathcal P_1(\R^{2N})$,  Lemma~\ref{visco:lemma4}-(3) ensures that the sequence $(\cT(m_n))_n$ is uniformly bounded. Moreover, Lemma~\ref{visco:lemma4}-(2) yields that the sequence $(\cT(m_n))_n$ is uniformly  bounded in $C^{1/2}([0,T];\cP_1(\R^{2N}))$ and  $L^\infty(0,T;\cP_2(\R^{2N}))$. 
By arguing as in the proof of Proposition~\ref{VV}, we obtain that, possibly for a subsequence (still denoted by~$\cT(m_n)$), $\cT(m_n)$ converges to some~$\mu$ in the $C([0,T];{\mathcal P}_1(\R^{2N}))$-topology.

2.
Theorem \ref{prp:m} ensures that, if $(u,m)$ is a solution of \eqref{eq:MFGAs},
 for any function $\psi\in C^0_b(\re^{2N})$,
\begin{equation}\label{reprfor}
\int_{\re^{2N}} \psi(x,v)\, m(x,v,t)dxdv=\int_{\re^{2N}}\psi(\overline{ \gamma}_{x,v}(t))m_0(x,v)\, dxdv
\end{equation}
where $\overline{\gamma}_{x,v}$ is the solution of (\ref{dyn}) (uniquely defined for a.e. $(x,v)\in\R^{2N}$).
\end{proof}

\begin{proof}[Proof of Proposition~\ref{prp:!}]
Let $(u_1,m_1)$ and $(u_2,m_2)$ be two solutions to system~\eqref{eq:MFGAs} in the sense of Definition~\ref{defsolmfg}. By Theorem~\ref{prp:m}, for $i=1,2$, the function~$m_i$ satisfies~\eqref{ambrosio} with $\overline {\gamma}_{x,v}$ replaced by $\overline {\gamma}_{x,v}^i$, which for a.e. $(x,v)$ is the solution to~\eqref{dyn} with~$u$ replaced by~$u_i$.

Moreover, let us recall from Lemma~\ref{L1}-(1) that the Lipschitz constant of $u_i$ has an at most linear growth in~$v$. By Gronwall's Lemma we obtain that $\overline{\gamma}_{x,v}^i$ is bounded. Since $m_0$ has compact support, we deduce the function~$m_i$ has compact support. In particular, we obtain that $\overline u:=u_1-u_2$ is an admissible test-function for the continuity equation satisfied by~$m_i$.

Taking advantage of the convexity of our Hamiltonian~$\mathcal{H}$ and of the monotonicity of the couplings~$F$ and~$G$, we can conclude the proof following the same arguments as in~\cite[Theorem 2.5]{ll07}.
\end{proof}
\section{Appendix}
Let us now consider the second order MFG system: for a positive number $\sigma$,
\begin{equation}\label{MFG2order}
\left\{
\begin{array}{lll}
(i)\ -\partial_t u-\sigma \Delta_{x,v} u-v\cdot D_xu+\frac1 2 \vert D_vu\vert^2-\frac1 2 \vert v \vert^2-l(x,v)  =F[m](x, v),\ & \textrm{in }\re^{2N}\times (0,T),\\
(ii)\ \partial_t m-\sigma \Delta_{x,v} m-\diver _v (m D_v u)-v\cdot D_xm=0, & \textrm{in }\re^{2N}\times (0,T),\\
(iii)\ m(x,v, 0)=m_0(x,v),\quad  u(x,v,T)=G[m(T)](x, v),& \textrm{on }\re^{2N}.
\end{array}\right.
\end{equation}
We aim at proving the existence and uniqueness of a classical solution to system~\eqref{MFG2order}.
We shall see that these results are byproducts of the estimates that we have already used above in the vanishing viscosity limit.  More
precisely, the properties obtained in Section~\ref{sect:c_eq} will play a crucial role in what follows.

\begin{theorem}\label{thm:2order}
Under our standing assumptions, there exists a classical solution to problem \eqref{MFG2order}. Moreover, if the coupling costs~$F$ and~$G$ satisfy~\eqref{monot}, the solution is unique. 
\end{theorem}

\begin{proof}
Our arguments are reminiscent of those used in the proof of~\cite[Theorem 3.1]{C}. We introduce
\begin{equation*}
  \mathcal{C}:=\left\{m\in C^0([0,T]; \mathcal{P}_1(\R^{2N})):\quad m(0)=m_0\right\}
\end{equation*}
which is a non-empty closed and convex subset of $C^0([0,T]; \mathcal{P}_1(\R^{2N}))$. We define a map $\mathcal{T}$ as follows: for any $m\in \mathcal{C}$, let $u$ be the unique solution to \eqref{MFG2order}-$(i)$ and $u(x,v,T)=G[m(T)](x, v)$ found in Lemma~\ref{visco:lemma5.2}; we set $\mathcal{T}(m)=\mu$ where $\mu$ is the unique solution to \eqref{MFG2order}-$(ii)$ and $m(x,v, 0)=m_0(x,v)$ found in Lemma~\ref{visco:buonapos}. Lemma \ref{visco:lemma4} ensures that $\mathcal{T}$ maps $\mathcal{C}$ into itself.\\
By the same arguments as in the proof of Theorem \ref{thm:main}, the map $\mathcal{T}$ is continuous with respect to the norm of $C([0,T]; \mathcal{P}_1(\R^{2N}))$ and it is compact. Hence, Schauder fixed point theorem ensures the existence of a fixed point $m$ for $\mathcal{T}$. Let $u$ denote the corresponding solution to \eqref{MFG2order}-$(i)$ and -$(iii)$.  By Lemma~\ref{visco:lemma5.2} and Lemma~\ref{visco:buonapos} again, $u$ and $m$ are regular. In conclusion, $(u,m)$ is the desired solution to \eqref{MFG2order}.\\
Let us now prove the uniqueness part of the statement. Let $(u_1,m_1)$ and $(u_2,m_2)$ be two solutions; set $\overline u=u_1-u_2$. Our aim is to follow the arguments in the proof of Proposition~\ref{prp:!}. To this end, it is enough to prove that $\overline u$ is an admissible test-function for $m_1$ and $m_2$.
Indeed, for any $R>1$, let $\phi_R$ be a cut-off function in $\R^{2N}$ defined by $\phi_R(x,v):=\phi_1(x/R,v/R)$ where $\phi_1$ is a $C^2$ function such that $\phi_1=1$ in $B_1$, $\phi_1=0$ outside $B_{2}$. Clearly,
\begin{equation}\label{uniq0}
D_{x,v}\phi_R=0 \quad\textrm{outside }\overline{B_{2R}\setminus B_R},\quad \|D_{x,v}\phi_R\|_\infty\leq C/R\quad \textrm{and }\|\Delta_{x,v}\phi_R\|_\infty\leq C/R^2.
\end{equation}
Using $\phi_R \overline u$ as test-function in~\eqref{MFG2order}-$(ii)$ with $(u,m)=(u_i,m_i)$ for $i=1$ or $i=2$, we get
\begin{eqnarray}\notag
0&=&\iint_{R^{2N}\times[0,T]} m\left[
-\phi_R \partial_t \overline u-\sigma \Delta_{x,v} (\phi_R \overline u)+ D_v u\cdot D_v(\phi_R \overline u)+v\cdot D_x(\phi_R \overline u)\right]\, dxdvdt\\\notag
&&+\int_{\R^{2N}}m(x,v,T)\phi_R \left(G[m_1(T)](x,v)-G[m_2(T)](x,v)\right)dxdv-\int_{\R^{2N}}m_0\phi_R \overline udxdv\\ \notag
&=&\iint_{R^{2N}\times[0,T]} m\phi_R\left(2 v\cdot D_x \overline u+\frac{|D_v u_2|^2-|D_v u_1|^2}2 +F[m_1] -F[m_2]+D_vu\cdot D_v\overline u \right)\, dxdvdt\\\notag
&&+\int_{\R^{2N}}m(x,v,T)\phi_R \left(G[m_1(T)](x,v)-G[m_2(T)](x,v)\right)dxdv-\int_{\R^{2N}}m_0\phi_R \overline udxdv\\\notag
&&+ \iint_{R^{2N}\times[0,T]} m\left(-\sigma \overline u \Delta_{x,v} \phi_R -2\sigma D_{x,v}\phi_R\cdot D_{x,v}\overline u\right)\, dxdvdt\\\label{uniq2}
&& +\iint_{R^{2N}\times[0,T]} m\left(\overline u D_v u\cdot D_v\phi_R +\overline u v\cdot D_x\phi_R\right)\, dxdvdt
\end{eqnarray}
where the second equality is due to equation \eqref{MFG2order}-$(i)$.
Since $m>0$ and $m\in L^\infty((0,T);\mathcal{P}_2(\R^{2N}))$ (see Lemma~\ref{visco:buonapos} and Lemma~\ref{visco:lemma4}), by the estimates on $u_i$ in Lemma \ref{visco:lemma5.2}, the dominated convergence theorem ensures that as $R\to\infty$, the first two lines in right hand side of \eqref{uniq2} converge to
\begin{multline*}
\iint_{R^{2N}\times[0,T]} m\left[2 v\cdot D_x \overline u-\frac{|D_v u_1|^2}2  +\frac{|D_v u_2|^2}2 +F[m_1] -F[m_2] +D_vu\cdot D_v\overline u\right]\, dxdvdt\\
+\int_{\R^{2N}}m(x,v,T) \left(G[m_1(T)](x,v)-G[m_2(T)](x,v)\right)dxdv-\int_{\R^{2N}}m_0 \overline udxdv;
\end{multline*}
hence, it remains to prove that the last two lines in the right hand side of~\eqref{uniq2} converge to $0$. Indeed, again by Lemmas~\ref{visco:lemma5.2},~\ref{visco:buonapos} and~\ref{visco:lemma4}, and by our estimates~\eqref{uniq0}, the dominated convergence theorem yields
\[
\iint_{R^{2N}\times[0,T]} m\left(-\sigma \overline u \Delta_{x,v} \phi_R -2\sigma D_{x,v}\phi_R\cdot D_{x,v}\overline u\right)\, dxdvdt\rightarrow 0.
\]
Let us now address to the last integral in the right hand side of~\eqref{uniq2}: the properties in~\eqref{uniq0} entail
\begin{equation*}
\left| m\left(\overline u D_v u\cdot D_v\phi_R +\overline u v\cdot D_x\phi_R\right)\right| \leq C  m(1+|v|^2) \chi_R
\end{equation*}
where $\chi_R$ is the characteristic function of $B_{2R}\setminus B_R$. Moreover, since $m\in L^\infty((0,T);\mathcal{P}_2(\R^{2N}))$, the right hand side in the last inequality belongs to $L^1$ independently of $R$. Therefore, again by the dominated convergence theorem, we get that as $R\to\infty$ the last integral in \eqref{uniq2} converges to $0$.
\end{proof}

\paragraph{\bf Acknowledgment.}
The authors are grateful to the anonymous referees for their fruitful comments and suggestions.
The work of Y. Achdou and N. Tchou was partially supported by the ANR (Agence Nationale de la Recherche) through MFG project ANR-16-CE40-0015-01.
Y. Achdou acknowledges support from the Chair Finance \& Sustainable Development and the FiME Lab  (Institut Europlace de Finance).
The work of N. Tchou is partially supported by the Centre Henri Lebesgue ANR-11-LABX-0020-01.
P. Mannucci and C. Marchi are members of GNAMPA-INdAM and were partially supported also by the research project of the University of Padova ``Mean-Field Games and Nonlinear PDEs'' and by the Fondazione CaRiPaRo Project ``Nonlinear Partial Differential Equations: Asymptotic Problems and Mean-Field Games''.

\bibliographystyle{plain}

\end{document}